\documentclass[11pt, reqno]{amsart}
\usepackage{amsmath,amssymb,amscd,mathrsfs,amscd}
\usepackage{graphics,verbatim}
\usepackage{amsmath,amstext,amsbsy,amssymb,amscd}
\usepackage{amsmath}
\usepackage{amsxtra}
\usepackage{amscd}
\usepackage{amsthm}
\usepackage{amsfonts}
\usepackage{graphicx}%
\usepackage{amssymb}
\usepackage{eucal}
\usepackage{color}
\usepackage{multirow}
\usepackage{graphicx}

\newcommand{\white}[1]{{\color{white}#1}}

\newtheorem{thm}[subsubsection]{Theorem}
\theoremstyle{plain}
\newtheorem{lem}[subsubsection]{Lemma}
\newtheorem{prop}[subsubsection]{Proposition}

\newtheorem{cor}[subsubsection]{Corollary}

\theoremstyle{definition}

\newtheorem{example}[subsubsection]{Example}

\theoremstyle{remark}

\definecolor{A}{rgb}{.75,1,.75}

\numberwithin{equation}{section}

\newcommand{\catO}{{\mathcal{O}}}

\newcommand{\A}{\mathcal A}

\newcommand{\Z}{\mathbb Z}
\newcommand{\N}{\mathbb N}

\newcommand{\la}{\lambda}

\newcommand{\om}{\omega}
\newcommand{\ov}{\overline}
\newcommand{\osp}{\mathfrak{osp}}

{\vskip-\lastskip\medskip
  \noindent
  {\em #1.}\enspace
  }%
{\qed\par\medskip
  }

\newcommand{\Q}{{\mathbb{Q}}}
\newcommand{\Qq}{{\Q(q)}}
\newcommand{\curlyA}{{\mathcal{A}}}

\newcommand{\curlyT}{{\mathcal{T}}}
\newcommand{\curlyC}{{\mathcal{C}}}

\newcommand{\ff}{{\bf f}}
\newcommand{\ffpr}{{}'{\bf f}}

\newcommand{\UU}{{\bf U}}
\newcommand{\Up}{\UU^+}
\newcommand{\Um}{\UU^-}
\newcommand{\Uz}{\UU^0}

\newcommand{\UUpr}{{\bf 'U}}

\newcommand{\bbinom}[2]{\begin{bmatrix}#1 \\ #2\end{bmatrix}}
\newcommand{\ir}{{\vphantom{|}_i r}}
\newcommand{\ri}{{r_i}}

\newcommand{\height}{{\operatorname{ht}}}

\newcommand{\zero}{{\bar{0}}}
\newcommand{\one}{{\bar{1}}}

\newcommand{\tK}{{\wtd{K}}}
\newcommand{\tJ}{{\wtd{J}}}

\newcommand{\wtd}{\widetilde}

\usepackage[all,knot]{xy}
\xyoption{arc}

\newcommand{\set}[1]{\left\{#1\right\}}
\newcommand{\parens}[1]{\left(#1\right)}
\newcommand{\ang}[1]{\left\langle#1\right\rangle}

\renewcommand{\bar}[1]{\overline{#1}}

\begin{document}

\title
[Foundations of quantum supergroups]{Quantum supergroups I.
Foundations}
\author[Clark, Hill and Wang]{Sean Clark, David Hill and Weiqiang Wang}
\address{Department of Mathematics, University of Virginia, Charlottesville, VA 22904}
\email{sic5ag@virginia.edu (S.~Clark), deh4n@virginia.edu
(D.~Hill), ww9c@virginia.edu (W.~Wang)}

\begin{abstract}
In this part one of a series of papers, we introduce a new version of quantum covering and super groups with
no isotropic odd simple root, which is suitable for the study of
integrable modules,  integral forms and the bar involution.
A quantum covering group involves parameters $q$ and $\pi$ with $\pi^2=1$,
and it specializes at $\pi=-1$ to
a quantum supergroup. Following
Lusztig, we formulate and establish various structural results of
the quantum covering groups, including a bilinear form, quasi-$\mathcal
R$-matrix, Casimir element, character formulas for integrable modules, 
and higher Serre relations.
\end{abstract}
\maketitle

\setcounter{tocdepth}{1}\tableofcontents


\section*{Introduction}

Quantum groups have been ubiquitous in Lie theory, mathematical
physics, algebraic combinatorics and low-dimensional topology
since their introduction by
Drinfeld and Jimbo \cite{Dr, Jim}. We refer to the books of Lusztig
and Jantzen \cite{Lu, J} for a systematic development of the structure
and representation theory of quantum groups.

In a recent paper \cite{HW} by two of the authors, the spin nilHecke
and quiver Hecke algebras (see Wang \cite{Wa},
Kang-Kashiwara-Tsuchioka \cite{KKT}, Ellis-Khovanov-Lauda
\cite{EKL}) were shown to provide a categorification of
quantum covering groups with a quantum parameter $q$ and a second
parameter $\pi$ satisfying $\pi^2=1$ (we refer to {\em loc. cit.}
for more references on categorification); a quantum covering group
specializes at $\pi=-1$ to  half of a quantum supergroup with no isotropic
odd simple roots, and to  half of the Drinfeld-Jimbo quantum group at $\pi=1$.


In the
rank one case, a version of the full quantum covering and super group for $\osp(1|2)$
suitable for constructing an integral form, as well as integrable
modules over $\Qq$ corresponding to each nonnegative integer, was formulated
by two of the authors \cite{CW}. In particular, the structure and representation
theories of quantum $\mathfrak{sl}(2)$ and quantum $\osp(1|2)$
were shown to be in a complete agreement, also see \cite{Zou} (in contrast to the classical
fact that there are ``fewer" integrable modules for
$\osp(1|2)$ than for $\mathfrak{sl}(2)$).

The goal of this paper is to lay the foundations of quantum covering
and super groups with no isotropic odd simple roots, following
Lusztig \cite[Part~I]{Lu} as a blueprint.
We define a new version of quantum covering and super groups with
no isotropic odd simple root, which is suitable for the study of
integrable modules for {\em all} possible dominant integral weights, exactly
as for the Drinfeld-Jimbo quantum groups.
We formulate and
establish various structural results of the quantum covering and super groups,
including a bilinear form, twisted derivations, integral forms,
bar-involution, quasi-$\mathcal R$-matrix, Casimir, characters for
integrable modules, and quantum (higher) Serre relations.

The results of this paper on quantum covering groups reduce to
Lusztig's quantum group setting \cite{Lu} when specializing the
parameter $\pi$ to $1$, and on the other hand, reduce to quantum
supergroup setting when specializing the parameter $\pi$ to $-1$.
For this reason, we work almost exclusively with quantum covering
groups. Even if one is mainly interested in the super case, writing
$\pi$ systematically for the super sign $-1$ offers a conceptual
explanation for various formulas and constructions. For earlier
definitions of quantum supergroups, we refer to Yamane \cite{Ya},
Musson-Zou \cite{MZ}, Benkart-Kang-Melville \cite{BKM}.

Let us describe the main results in detail.  As
in \cite{Kac}, a super Cartan datum is a Cartan datum $(I,\cdot)$
with a partition $I=I_\zero \sqcup I_\one$ subject to some natural
conditions; also see \cite{HW}.
Note the only finite type super Cartan datum is of type $B(0,n)$,
for $n\ge 1$. In Section~\ref{sec:algebraf}, we formulate the definition 
of  half a quantum covering group associated to a super Catan datum.
We develop the
properties of a bilinear form (and a dual version) and twisted
derivations on half the quantum covering group systematically. Then
we provide a new proof using twisted derivations of a theorem in
\cite{HW} (also cf. Yamane \cite{ Ya} and Geer \cite{Gr}) that the existence of a
non-degenerate bilinear form implies the quantum Serre relations.

Motivated by the rank one construction in \cite{CW}, we formulate in
Section~\ref{sec:algebraU} a new version of quantum super and
covering groups with generators $E_i, F_i, K_\mu$, and
additional generators $J_\mu$, for $i \in I$ and $\mu\in Y$ (the co-weight lattice). The new generators $J_i$
play a crucial
role in formulating the notion of integrable modules of a quantum
supergroup for {\em all} dominant integral weights. A study of all such representations was not
possible before (cf. \cite{Kac, BKM}).

In Section~\ref{sec:quasiR}, we formulate the quasi-$\mathcal
R$-matrix for quantum covering or super groups and establish its
basic properties. This generalizes the construction in the rank one
case in \cite{CW}. Then we construct the quantum Casimir and use it
to prove the complete reducibility of the integrable modules. We
show that the simple integrable modules are parametrized by $\pi=\pm 1$ and the
dominant integral weights (in contrast to \cite{BKM, Kac}), and
their character formulas coincide with their
counterpart for quantum groups (which was established by Lusztig \cite{Lu1}).
This character formula  (in case $\pi=-1$) is shown to
hold for the irreducible integrable modules under some ``evenness"
restrictions on highest weights as in \cite{BKM} (where a definition
of quantum supergroups without operators $J_i$ was used),
deforming the construction in \cite{Kac}.

The higher Serre
relations for quantum covering groups are then established in
Section~\ref{sec:higherSerre}.

%

This paper lays the foundation for further studies of quantum covering and super groups.
In a sequel \cite{CHW}, we will construct the canonical basis, \`a
la Lusztig and Kashiwara, of quantum covering groups and of
integrable modules. In yet another paper, a braid group action
on a quantum covering group and its integrable modules will be studied in depth.

\vspace{.2cm}  {\bf Acknowledgement.}
The third author was partially supported by NSF DMS-1101268.

\section{The algebra $\bf f$}  \label{sec:algebraf}

In this section, starting with the super Cartan datum and root datum,
we formulate half a quantum covering group $\bf f$ in terms of a bilinear form
on a free superalgebra $\ffpr$, and show that the $(q,\pi)$-Serre relations are
satisfied in $\bf f$.

\subsection{Super Cartan datum}
 \label{subsec:Cartan}

A {\em Cartan datum} is a pair $(I,\cdot)$ consisting of a finite
set $I$ and a symmetric bilinear form $\nu,\nu'\mapsto \nu\cdot\nu'$
on the free abelian group $\Z[I]$ with values in $\Z$ satisfying
\begin{enumerate}
 \item[(a)] $d_i=\frac{i\cdot i}{2}\in \Z_{>0}$;

  \item[(b)]
$2\frac{i\cdot j}{i\cdot i}\in -\N$ for $i\neq j$ in $I$, where $\N
=\{0,1,2,\ldots\}$.
\end{enumerate}
If the datum can be decomposed as $I=I_\zero\coprod I_\one$ such that
\begin{enumerate}
        \item[(c)] $I_\one\neq\emptyset$,
        \item[(d)] $2\frac{i\cdot j}{i\cdot i} \in 2\Z$ if $i\in I_\one$,
\end{enumerate}
then it is called a {\em super Cartan datum}.

The $i\in I_\zero$ are called even, $i\in I_\one$ are called odd. We
define a parity function $p:I\rightarrow\set{0,1}$ so that $i\in
I_{\bar{p(i)}}$. We extend this function to the homomorphism
$p:\Z[I]\rightarrow \Z$. Then $p$ induces a $\Z_2$-grading on
$\Z[I]$ which we shall call the parity grading.
We define the {\em height} of $\nu=\sum_{i\in I}\nu_i i\in \Z[I]$ by $\height(\nu)=\sum \nu_i$. (Note we use different notation than \cite{Lu},
where the same quantity is denoted by ${\rm tr}(\nu)$.)

A super Cartan datum $(I,\cdot)$ is said to be of {\em finite}
(resp. {\em affine}) type exactly when $(I,\cdot)$ is of
finite (resp. affine) type as a Cartan datum (cf. \cite[\S 2.1.3]{Lu}).
In particular, from (a) and (d) 
we see that the only super Cartan datum of finite type is the one
corresponding to the Lie superalgebras of type
$B(0,n)$ for $n\geq 1$.

A super Cartan datum is called  {\em bar-consistent} or simply {\em
consistent} if it satisfies
\begin{enumerate}
        \item[(e)] $d_i\equiv p(i) \mod 2, \quad \forall i\in I.$
\end{enumerate}
We note that (e) is almost always satisfied for
super Cartan data of finite or affine type (with one exception). A
super Cartan datum is not assumed to be (bar-)consistent unless
specified explicitly below. (Roughly speaking, the ``bar-consistent"
condition is imposed whenever a bar involution is involved later
on.)

Note that (d) and (e) imply that
\begin{enumerate}
        \item[(f)] $i\cdot j\in 2\Z$ for all $i,j\in I$.
\end{enumerate}

\subsection{Root datum}\label{subsec:rootdatum}

A {\em root datum} associated to a super Cartan datum $(I,\cdot)$
consists of
\begin{enumerate}
\item[(a)]
two finitely generated free abelian groups $Y$, $X$ and a
perfect bilinear pairing $\ang{\cdot, \cdot}:Y\times X\rightarrow \Z$;

\item[(b)]
an embedding $I\subset X$ ($i\mapsto i'$) and an embedding $I\subset
Y$ ($i\mapsto i$) satisfying

\item[(c)] $\ang{i,j'}=\frac{2 i\cdot j}{i\cdot i}$ for all $i,j\in I$.
\end{enumerate}
We will always assume that the image of the imbedding $I\subset X$
(respectively, the image of the imbedding $I\subset Y$) is linearly
independent in $X$ (respectively, in $Y$).

Let $X^+=\set{\lambda\in X\mid \ang{i,\lambda}\in \N \text{ for all } i
\in I}$. Note that there are no additional ``evenness'' assumptions for $X^+$.

Let $\pi$ be a parameter such that
$$\pi^2=1.
$$
For any $i\in I$, we set
$$q_i=q^{i\cdot i/2}, \qquad \pi_i=\pi^{p(i)}.
$$
Note that when the datum is consistent, $\pi_i=\pi^{\frac{i\cdot
i}{2}}$; by induction, we therefore have
$\pi^{p(\nu)}=\pi^{\nu \cdot \nu/2}$ for $\nu\in \Z[I]$. 
We extend this notation so that if
$\nu=\sum \nu_i i\in \Z[I]$, then
$$
q_\nu=\prod_i q_i^{\nu_i}, \qquad \pi_\nu=\prod_i \pi_i^{\nu_i}.
$$
For any ring $R$ we define a new ring $R^\pi =R[\pi]/(\pi^2-1)$ 
(with $\pi$ commuting with $R$).
We shall need $\Qq^\pi$ below.


%
\subsection{Braid group and Weyl group}

Assume a Cartan (super) datum $(I, \cdot)$ is given. For $i\neq j
\in I$ such that $\ang{i,j'} \ang{j,i'}>0$, we define an integer
$m_{ij} \in \Z_{\ge 2}$ by $\cos^2 \frac{\pi}{m_{ij}} =
\frac14\ang{i,j'} \ang{j,i'}$ if it exists, and set $m_{ij}=\infty$ otherwise. We have
\begin{center}\begin{tabular}{c|ccccc}
$\ang{i,j'} \ang{j,i'}$&0&1&2&3&$\geq4$\\\hline
$m_{ij}$&2&3&4&6&$\infty$
\end{tabular}\end{center}

The braid group (associated to $I$) is the group generated by $s_i$
$(i\in I)$ subject to the relations (whenever $m_{ij}<\infty$):
\begin{align} \label{eq:BraidRelations}
\underbrace{s_is_js_i\cdots}_{m_{ij}}\,=\,\underbrace{s_js_is_j\cdots}_{m_{ij}},
\end{align}
The Weyl group $W$ is defined to be the group generated  by $s_i$
$(i\in I)$ subject to relations \eqref{eq:BraidRelations} 
and additional relations
$s_i^2=1$ for all $i$.

For $i\in I$, we let $s_i$ act on $X$ (resp. $Y$) as follows: for
$\la \in X, \la^\vee \in Y$,
$$
s_i(\lambda)=\lambda-\langle i,\lambda\rangle i',\qquad
s_i(\lambda^\vee)=\lambda^\vee-\langle \lambda^\vee,i'\rangle i.
$$
This defines actions of the Weyl group $W$ on $X$ and $Y$.

\subsection{The algebras $\ffpr$ and $\bf f$}

Define $\ffpr$ to be the free associative $\Qq^\pi$-superalgebra
with $1$ and with even generators $\theta_i$ for $i\in I_\zero$ and
odd generators $\theta_i$ for $i\in I_\one$. We abuse notation and
define the parity grading on $\ffpr$ by $p(\theta_i)=p(i)$.
We also have a weight grading $|\cdot|$ on $\ffpr$ defined
by setting $|\theta_i|=i$.

The tensor product 
$\ffpr \otimes \ffpr$ as a $\Qq^\pi$-superalgebra has the
multiplication
$$
(x_1 \otimes x_2)(x_1'\otimes x_2') =q^{|x_2|\cdot |x_1'|}
\pi^{p(x_2)p(x_1')} x_1x_1' \otimes x_2 x_2'.
$$

{\bf Here and below, in all displayed formulas, we will implicitly
assume the elements involved are $\N[I]\times \Z_2$-homogeneous. }

There is a similar multiplication formula in $\ffpr \otimes \ffpr \otimes \ffpr$:
\begin{align*}
(x_1 \otimes &  x_2 \otimes x_3)(x_1'\otimes x_2'\otimes x_3')
  \\
&=q^{|x_2|\cdot |x_1'| +|x_3|\cdot |x_2'| +|x_3|\cdot |x_1'|}
\pi^{p(x_2)p(x_1') +p(x_3)p(x_2') +p(x_3)p(x_1')} x_1x_1' \otimes
x_2 x_2' \otimes x_3 x_3'.
\end{align*}

We will take $r: \ffpr \rightarrow \ffpr\otimes\ffpr$ to be an
algebra homomorphism such that $r(\theta_i) =\theta_i \otimes 1
+1\otimes \theta_i$ for all $i\in I$. One checks the following
co-associativity holds:

$(r\otimes 1) r =(1\otimes r)r: \ffpr \rightarrow\ffpr \otimes \ffpr
\otimes \ffpr$;  this is an algebra homomorphism.

\begin{prop}\label{P:bilinearform}
There exists a unique bilinear form $(\cdot,\cdot)$ on $\ffpr$
with values in $\Q$ such that $(1,1)=1$ and
\begin{enumerate}
\item[(a)]
$(\theta_i, \theta_j) = \delta_{ij} (1-\pi_i q_i^{-2})^{-1} \quad
(\forall i,j\in I);$

\item[(b)]
$(x, y'y'') = (r(x), y'\otimes y'') \quad (\forall x,y',y'' \in
\ffpr);$

\item[(c)]
$(xx', y'') = (x\otimes x', r(y'')) \quad (\forall x,x',y'' \in
\ffpr).$
\end{enumerate}
Moreover, this bilinear form is symmetric.
\end{prop}
Here, the induced bilinear form  $(\ffpr \otimes \ffpr) \times
(\ffpr \otimes \ffpr) \rightarrow \Q(q)$ is given by
\begin{equation}  \label{eq:tensorform}
(x_1\otimes x_2, x_1' \otimes x_2') :=
(x_1,x_1')(x_2, x_2'),
\end{equation}
for homogeneous $x_1,x_2, x_1',x_2' \in\ffpr$.

This is basically \cite[Proposition~3.3]{HW}, where $(\theta_i,
\theta_j) = \delta_{ij} (1-\pi_i q_i^{2})^{-1}$ was imposed (note a
different sign on the exponent for $q_i^2$). These two cases do not
exactly match under the bar-involution (which sends $q \mapsto \pi
q^{-1}$), and so we redo a careful proof here.

\begin{proof}
We follow \cite[1.2.3]{Lu} to define an associative algebra
structure on $\ffpr^* :=\oplus_{\nu} \ffpr^*_\nu$ by transposing the
``coproduct" $r: \ffpr \rightarrow \ffpr \otimes \ffpr$. In particular,
for $g,h\in \ffpr^*$, we define $gh(x):=(g\otimes h )(r(x))$,
where $(g\otimes h)(y\otimes z)=g(y)h(z)$.

Let $\xi_i\in \ffpr^*_i$ be defined by $\xi_i(\theta_i) =(1-\pi_i
q_i^{-2})^{-1}.$ Let $\phi: \ffpr \rightarrow \ffpr^*$ be the unique
algebra homomorphism such that $\phi(\theta_i) =\xi_i$ for all $i$.
The map $\phi$ preserves the $\N[I]\times \Z_2$-grading.

Define $(x,y) =\phi(y)(x)$, for $x,y\in \ffpr$. The properties (a) and (b)
follow directly from the definition.

Clearly $(x,y)=0$ unless (homogeneous) $x,y$ have the same weight in
$\N[I]$ and the same parity. All elements involved below will be
assumed to be homogeneous.


It remains to prove (c). Assume that (c) is known for $y''$ replaced
by $y$ or $y'$ and for any $x, x'$. We then prove that (c) holds for
$y''=yy'$. Write
\begin{align*}
r(x) =\sum x_1\otimes x_2, \quad & r(x') =\sum x_1'\otimes x_2',
  \\
r(y) =\sum y_1\otimes y_2, \quad & r(y') =\sum y_1'\otimes y_2'.
\end{align*}
Then
\begin{align*}
r(xx') &=\sum q^{|x_2|\cdot|x_1'|}\pi^{p(x_2)p(x_1')} x_1x_1'\otimes
x_2x_2',
  \\
r(yy') &=\sum q^{|y_2|\cdot|y_1'|}\pi^{p(y_2)p(y_1')} y_1y_1'\otimes
y_2y_2'.
\end{align*}
We have
\begin{align}
&(xx',yy') =(\phi(y)\phi(y'))(xx')=(\phi(y)\otimes\phi(y'))(r(xx'))
   \notag \\
&= \sum q^{|x_2|\cdot |x_1'|} \pi^{p(x_2)p(x_1')}
(x_1x_1',y)(x_2x_2',y')
   \notag \\
&= \sum q^{|x_2|\cdot |x_1'|} \pi^{p(x_2)p(x_1')}
(x_1 \otimes x_1',r(y))(x_2\otimes x_2',r(y'))
   \notag \\
&= \sum q^{|x_2|\cdot |x_1'|}
\pi^{p(x_2)p(x_1')}
(x_1,y_1)(x_1',y_2)(x_2,y_1')(x_2',y_2').
\label{eq:bilformd}
\end{align}

On the other hand,
\begin{align}
&(x\otimes x', r(yy'))
 =\sum q^{|y_2|\cdot|y_1'|}\pi^{p(y_2)p(y_1')}
(x\otimes x', y_1y_1'\otimes y_2y_2')
   \notag \\
&=\sum q^{|y_2|\cdot|y_1'|}\pi^{p(y_2)p(y_1')}
(x,y_1y_1')(x',y_2y_2')
   \notag \\
&=\sum q^{|y_2|\cdot|y_1'|}\pi^{p(y_2)p(y_1')}
(r(x),y_1\otimes y_1')(r(x'),y_2\otimes y_2')
   \notag \\
&=\sum q^{|y_2|\cdot|y_1'|}
\pi^{p(y_2)p(y_1')}
(x_1,y_1)(x_1',y_2)(x_2,y_1')(x_2',y_2').
\label{eq:bilforme}
\end{align}
For a summand to make nonzero contribution, we may assume that each
of the four pairs $\{x_1,y_1\}, \{x_1',y_2\}, \{x_2,y_1'\},
\{x_2',y_2'\}$ have the same weight in $\N[I]$ and the same parity.
One checks that the powers of $q$ and $\pi$ in \eqref{eq:bilformd} and
\eqref{eq:bilforme} match perfectly. Hence the two sums in \eqref{eq:bilformd} and
\eqref{eq:bilforme} are equal, and whence (c).
\end{proof}


We set $\mathcal{I}$ to denote the radical of $(\cdot, \cdot)$. As in
\cite{Lu}, this radical is a 2-sided ideal of $\ffpr$.


Let $\ff=\ffpr/\mathcal{I}$ be the quotient algebra of $\ffpr$ by
its radical. Since the different weight spaces are orthogonal with
respect to this inner product, the weight space decomposition
descends to a decomposition $\ff=\bigoplus_{\nu} \ff_\nu$ where
$\ff_\nu$ is the image of $\ffpr_\nu$. Each weight space is finite
dimensional. The bilinear form descends to a bilinear form on $\ff$
which is non-degenerate on each weight space.

Note that the notation of $\ffpr$ and $\bf f$ in this paper
corresponds to the notation of $\ffpr^\pi$ and $\bf f^\pi$ in
\cite{HW}.

The map $r:\ffpr\rightarrow \ffpr\otimes\ffpr$ satisfies
$r(\mathcal I)\subset \mathcal I\otimes \ffpr + \ffpr\otimes \mathcal I$
(the proof being entirely the same as in \cite[\S 1.2.6]{Lu}),
whence it descends to a 
well-defined homomorphism $r:\ff\rightarrow \ff\otimes \ff$.

Let
${}^t r: \ffpr \rightarrow \ffpr\otimes\ffpr$ be the composition of
$r$ with the permutation map \[x\otimes y \mapsto y\otimes x\] of
$\ffpr\otimes\ffpr$ to itself. (To have the signs work out below,
the tensor permutation cannot be signed.)

The anti-involution $\sigma: \ffpr \rightarrow \ffpr$ satisfies
$\sigma(\theta_i)=\theta_i$ for each $i\in I$ and
$$
\sigma (xy) = \sigma(y)\sigma(x).
$$

\begin{lem} 
\begin{enumerate}
\item[(a)]
We have $r(\sigma(x)) = (\sigma \otimes \sigma) {}^t r(x)$, for all $x\in \ffpr$.

\item[(b)]
We have $(\sigma(x),\sigma(x')) =(x, x')$ for all $x, x'\in \ffpr$.
\end{enumerate}
\end{lem}

\begin{proof}
Since (b) will follow immediately from (a), it suffices to prove that
$r(\sigma(x)) = (\sigma \otimes \sigma) {}^t r(x)$, for all $x\in \ffpr$.
This is obviously true for $x\in\set{1,\theta_i: i\in I}$.

Suppose that $r(\sigma(x')) = (\sigma \otimes \sigma) {}^t r(x')$
and $r(\sigma(x'')) = (\sigma \otimes \sigma) {}^t r(x'')$. Let
$r(x')=\sum x_1'\otimes x_2'$ and $r(x'')=\sum x_1''\otimes x_2''$.
Then $r(x'x'')=\sum
q^{|x_2'||x_1''|}\pi^{p(x_2')p(x_1'')}x_1'x_1''\otimes x_2'x_2''$
and we have
\begin{align}
r(\sigma(x'x''))&=r(\sigma(x''))r(\sigma(x'))\notag\\
&=\parens{\sum  \sigma(x_2'')\otimes \sigma(x_1'')}
\parens{\sum \sigma(x_2')\otimes \sigma(x_1')}\notag\\
&=\sum \pi^{p(x_2')p(x_1'')}q^{|x_2'||x_1''|}
\sigma(x_2'x_2'')\otimes \sigma(x_1'x_1'')=\sigma\otimes \sigma
(^tr(x'x'')).
 \notag
\end{align}
The lemma is proved.
\qed\end{proof}

We note that $\sigma$ descends to $\ff$ and shares the above properties.


Let $\ov{\phantom{r}}:\Qq^\pi\rightarrow\Qq^\pi$ be the unique
$\Q$-algebra involution (called the bar involution) 
satisfying $\ov{q}= \pi q^{-1}$ and $\ov{\pi}=
\pi$.

Assume the super Cartan datum is consistent.
Then
\begin{align}\label{eq:barinvqi}
\bar{q_i}=\pi_i
q_i^{-1}.
\end{align}
We define a bar involution $\ov{\phantom{r}}:\ffpr\rightarrow\ffpr$
such that $\ov{\theta_i} =\theta_i$ for all $i\in I$ and $\ov{fx}
=\ov{f}\ov{x}$ for $f\in\Qq^\pi$ and $x\in \ffpr$.

Let $\ffpr \bar{\otimes} \ffpr$ be the $\Qq^\pi$-vector space $\ffpr
\otimes \ffpr$ with multiplication given by
$$
(x_1 \otimes x_2)(x_1'\otimes x_2') =(\pi q^{-1})^{|x_2|\cdot
|x_1'|} \pi^{p(x_2)p(x_1')} x_1x_1' \otimes x_2 x_2'.
$$

Define $\ov{r}$ still by $\ov{r}(x) =\ov{r(\ov{x})}$. Then
$\ov{r}:\ffpr \rightarrow \ffpr \bar{\otimes} \ffpr$ is an algebra
homomorphism, being a composition of homomorphisms.

The co-associativity holds for $\ov{r}$:
\[
(\ov r \otimes 1)(\ov r (x))=\ov{(r\otimes
1)r(\bar{x})}=\ov{(1\otimes r)r(\bar{x})}=(1 \otimes \ov r)(\ov r
(x)).
\]
By checking on the algebra generators $\theta_i$, it is an easy
computation to see that this is an algebra homomorphism.

Let $\{\cdot,\cdot\}:\ffpr \times \ffpr \rightarrow \Q(q)$ be the
symmetric bilinear form defined by
$$
\{x,y\} =\ov{(\ov{x}, \ov{y})}.
$$
It satisfies: $\{1,1\}=1$, and
\begin{align*}
\{\theta_i,\theta_j\} &=\delta_{i,j} (1- \pi_iq_i^2)^{-1};
 \\
\{x, y'y''\} &=\{\ov{r} (x), y'\otimes y''\},  \text{ for all }x, y', y'' \in \ffpr.
\end{align*}
\begin{lem} \label{L:curlyform}
Assume the super Cartan datum is consistent.
\begin{enumerate}
\item[(a)]
Let $r(x) = \sum x_1\otimes x_2$. We have
$$\ov{r}(x)
=\sum (\pi q)^{-|x_1|\cdot |x_2|} \pi^{p(x_1)p(x_2)} x_2\otimes x_1.$$

\item[(b)]
$\{x,y\}=(-1)^{\height |x|} \pi^{\frac{p(x)p(y)+p(x)}{2}}
q^{-|x|\cdot |y|/2}q_{-|x|}(x,\sigma(y)).$
\end{enumerate}
\end{lem}

\begin{proof}
It is straightforward to check both claims are true when
$x=\theta_i$ and $y=\theta_j$ for some $i,j\in I$.

Assume (a) holds for $x$ replaced by $x'$ and by $x''$. We shall
prove the claim for $x=x'x''$.

Recall $\ov{q} =\pi q^{-1}$, and $r(\ov{x}) =\ov{\ov{r}(x)}.$
Write
\begin{align}
r(x') &=\sum x_1'\otimes x_2', \quad  r(x'') =\sum x_1''\otimes x_2'',
  \notag\\
r(x'x'') &=\sum q^{|x_1''|\cdot |x_2'|} \pi^{p(x_1'')p(x_2')} x_1'x_1''\otimes x_2'x_2''.\label{eq:curlyformc}
\end{align}

By assumption, we have
\begin{align*}
r(\ov{x'}) &=\sum q^{|x_1'|\cdot |x_2'|} \pi^{p(x_1')p(x_2')} \ov{x_2'}\otimes\ov{x_1'},
  \\
r(\ov{x''}) &=\sum q^{|x_1''|\cdot |x_2''|} \pi^{p(x_1'')p(x_2'')} \ov{x_2''}\otimes\ov{x_1''}.
\end{align*}
Hence,
\begin{align*}
r(\ov{x'}) r(\ov{x''})
 &=\sum q^{|x_1'|\cdot |x_2'|} \pi^{p(x_1')p(x_2')}
 q^{|x_1''|\cdot |x_2''|} \pi^{p(x_1'')p(x_2'')}
  (\ov{x_2'}\otimes\ov{x_1'})(\ov{x_2''}\otimes\ov{x_1''})
  \\
 &=\sum q^{|x_1'|\cdot |x_2'|+|x_1''|\cdot |x_2''|}
  \pi^{p(x_1')p(x_2')+p(x_1'')p(x_2'')+p(x_1')p(x_2'')}
 q^{|x_1'|\cdot |x_2''|}
 \ov{x_2'}\ov{x_2''}\otimes\ov{x_1'}\ov{x_1''}.
\end{align*}
Then,
\begin{align}
\ov{r}(x'x'') &=\ov{r(\ov{x'})r(\ov{x''})}
  \notag\\
 &=\sum (\pi q)^{-(|x_1'|\cdot |x_2'|+|x_1''|\cdot |x_2''|+|x_1'|\cdot |x_2''|)}
  \pi^{p(x_1')p(x_2')+p(x_1'')p(x_2'')+p(x_1')p(x_2'')}
 {x_2'}{x_2''}\otimes {x_1'}{x_1''}
  \notag\\
 &=\sum (\pi q)^{-|x_1'x_1''|\cdot |x_2'x_2''|}
  \pi^{p(x_1'x_1'')p(x_2'x_2'')}
  q^{|x_1''|\cdot |x_2'|}\pi^{p(x_1'')p(x_2')+|x_1''|\cdot |x_2'|}
 {x_2'}{x_2''}\otimes {x_1'}{x_1''}.\notag\\
 \intertext{Now, since the datum is consistent,
 $|x_1''|\cdot |x_2'|\in 2\Z$, and hence we have}
 \ov{r}(x'x'') &=\sum (\pi q)^{-|x_1'x_1''|\cdot |x_2'x_2''|}
  \pi^{p(x_1'x_1'')p(x_2'x_2'')}
  q^{|x_1''|\cdot |x_2'|}\pi^{p(x_1'')p(x_2')}
 {x_2'}{x_2''}\otimes {x_1'}{x_1''}.\label{eq:curlyformd}
\end{align}
Comparing \eqref{eq:curlyformc} and \eqref{eq:curlyformd}, we see that (a) holds.

Let $\mathcal{S}$ be the set of $y\in \ffpr$ such that (b) holds for all $x\in \ffpr$.
Let $y',y''\in \mathcal S$; we will show $y=y'y''\in \mathcal S$
Let $x\in \ffpr$ and write $r(x)=\sum x'\otimes x''$ with $x,x''$ homogeneous.
Then
\begin{align*}
\set{x,y'y''} &=\set{\ov{r}(x), y'\otimes y''}
= \set{\sum (\pi q)^{-|x'|\cdot |x''|} \pi^{p(x')p(x'')}x''\otimes x', y'\otimes y''}\\
&=\sum q^{-|x'|\cdot |x''|} \pi^{p(x')p(x'')} \set{x'',y'}\set{x',y''}\\
&=\sum(-1)^{\height |x'|+\height |x''|}
q^{\frac{-|x''|\cdot |y'|-|x'|\cdot |y''|-2|x'|\cdot |x''|}{2}}q_{-|x'|-|x''|}\\
&\qquad*\pi^{p(x')p(x'')+\frac{p(x')p(y'')+p(x')}{2}+\frac{p(x'')p(y')-p(x'')}{2}}
 (x'',\sigma(y'))(x',\sigma(y''))\\
&\stackrel{(\dagger)}{=}\sum (-1)^{\height|x|}q^{\frac{-|x|\cdot |y|}{2}}q_{-|x|}
\pi^{\frac{p(x)p(y)+p(x)}{2}}(x'\otimes x'',\sigma(y'')\otimes \sigma(y'))\\
&=(-1)^{\height|x|}q^{\frac{-|x|\cdot |y|}{2}}q_{-|x|}\pi^{\frac{p(x)p(y)+p(x)}{2}}(x,\sigma(y'y''))
\end{align*}
where the equality $(\dagger)$ follows from the observation that the
nonzero terms in the sum only occur when the each of the pairs
$\set{x',y''}$ and $\set{x'', y'}$ are of the same weight and
parity.  Therefore we see $y\in \mathcal S$. Since the algebra generators lie
in $\mathcal S$, the claim is proved.
\qed\end{proof}

In particular, we observe the following corollary.

\begin{cor}
Assume the super Cartan datum is consistent.
Then $\ov{\phantom{x}}\ $ descends to an involution on $\ff$.
\end{cor}

\subsection{The maps $r_i$ and $_i r$}

Let $i\in I$. Clearly there are unique $\Qq^\pi$-linear maps
$\ri,\ir:\ffpr\rightarrow\ffpr$ such that $\ri(1)=\ir(1)=0$ and
$\ri(\theta_j)=\ir(\theta_j)=\delta_{ij}$ satisfying
\[
\ir(xy)=\ir(x)y+\pi^{p(x)p(i)}q^{|x|\cdot i}x\ir(y)
\]
\[
\ri(xy)=\pi^{p(y)p(i)}q^{|y|\cdot i}\ri(x)y+x\ri(y)
\]
for homogeneous $x,y\in\ffpr$, see \cite{K}. We see that if $x\in \ffpr_\nu$, then
$\ir(x),\ri(x)\in\ffpr_{\nu-i}$ and moreover that
\begin{equation}  \label{eq:derivcoprod}
r(x)=\ri(x)\otimes \theta_i+\theta_i\otimes \ir(x)+(...)
\end{equation}
where $(...)$ stands in for other bi-homogeneous terms $x'\otimes x''$
with $|x'|\neq i$ and $|x''|\neq i$. Therefore, we have
\begin{equation}  \label{eq:derivadjunct}
(\theta_i y, x) = (\theta_i,\theta_i)(y, \ir(x)), \quad
(y \theta_i, x) = (\theta_i,\theta_i)(y, \ri(x))
\end{equation}
for all $x,y\in \ffpr$, so $\ir(\mathcal{I})\cup
\ri(\mathcal{I})\subseteq \mathcal{I}$. Hence, both maps descend to
maps on $\ff$. It is also easy to check that
\[\ri\, \sigma=\sigma\, \ir.\] Indeed, this is trivially true for the
generators, and if this holds for $x,y\in \ff$, then
\begin{align*}
\ri \, \sigma(xy)&=\ri(\sigma(y)\sigma(x))
=\pi^{p(i)p(x)}q^{i\cdot |x|} \ri(\sigma(y))\sigma(x)+\sigma(y)\ri(\sigma(x))\\
&=\sigma(\pi^{p(i)p(x)}q^{i\cdot |x|} x\ir(y)+\ir(x)y)=\sigma\,
\ir(xy).
\end{align*}

\begin{lem}
Assume $(I,\cdot)$ is consistent. For any homogeneous $x\in\ff$, we
have \[\ri(x)=\pi^{p(x)p(i)-p(i)p(i)}q^{|x|\cdot i-i\cdot
i}\ov{\ir(\ov{x})}.
\]
\end{lem}
\begin{proof}
This is trivial when $x=\theta_i$. Now assume this is true for $x,y\in \ffpr$. Then
\begin{align*}
\ov{\ir(\ov{xy})}&=\ov{\ir(\ov{x})}y+\pi^{p(x)p(i)}(\pi q)^{-|x|\cdot i}\ov{\ir(\ov{y})}\\
&=\pi^{-p(x)p(i)+p(i)p(i)}q^{-|x|\cdot i+i\cdot i}
\ri(x)y \\
 &\quad +\pi^{-p(y)p(i)+p(i)p(i)}q^{-|y|\cdot i+i\cdot i}\pi^{p(x)p(i)}(\pi q)^{-|x|\cdot i}x\ri(y)\\
&=\pi^{-p(x+y)p(i)+p(i)p(i)}q^{-|x+y|\cdot i+i\cdot i}\parens{\pi^{p(y)p(i)}q^{|y|\cdot i}\ri(x)y+x\ri(y)}\\
&=\pi^{-p(x+y)p(i)+p(i)p(i)}q^{-|x+y|\cdot i+i\cdot i}\ri(xy).
\end{align*}
The lemma is proved.
\qed\end{proof}

\begin{lem}\label{lem:derivkill}
Let $x\in \ff_\nu$ where $\nu\in \N[I]$ is nonzero.
\begin{enumerate}
\item[(a)] If $\ri(x)=0$ for all $i\in I$, then $x=0$.
\item[(b)] If $\ir(x)=0$ for all $i\in I$, then $x=0$.
\end{enumerate}
\end{lem}
\begin{proof}
Suppose that $\ri(x)=0$ for all $i$.
Using  \eqref{eq:derivadjunct}, this means that $(y\theta_i,x)=0$ for all $y\in \ff$
and all $i\in I$. But since $\ff$ is spanned by monomials in the $\theta_i$, this implies
$x\in\mathcal{I}$, and so $x=0$ in $\ff$. The proof of (b) proceeds similarly.
\qed\end{proof}

\subsection{Gaussian  $(q,\pi)$-binomial coefficients}
\label{sec:gaussids}

Let $\A=\Z[q,q^{-1}]$, and let $\A^\pi$ be as in $\S$1.1.3. For $a\in \Z$ and $t\in \N$, we define the
{\em $(q,\pi)$-binomial coefficients} to be

\begin{align}
\bbinom{a}{t}_i&=\frac{\prod_{s=0}^{t-1}
((\pi_iq_i)^{a-s}-q_i^{s-a})}{\prod_{s=1}^{t}((\pi_iq_i)^{s}-q_i^{-s})}. 
\notag \\
\intertext{We have }
\label{eq:binomida} \bbinom{a}{t}_i&=(-1)^t\pi_i^{ta-\binom{t}{2}}\bbinom{t-a-1}{t}_i, \\
\label{eq:binomidb} \bbinom{a}{t}_i&=0\qquad\text{if }0\leq a <t, \\
\label{eq:binomidc} \prod_{j=0}^{a-1}&\parens{1+(\pi_i q_i^{2})^jz}
=\sum_{t=0}^a\pi_i^{\binom{t}{2}}q_i^{t(a-1)}\bbinom{a}{t}_i z^t
 \quad \text{if } a\geq 0.
\end{align}
Here $z$ is another indeterminate. From \eqref{eq:binomida} and \eqref{eq:binomidc} we deduce that
\begin{equation}\label{eq:binomidd}
\bbinom{a}{t}_i\in\ \A.
\end{equation} 
If $a', a''$ are integers and
$t\in\N$, then
\begin{equation}\label{eq:binomide}
\bbinom{a'+a''}{t}_i=\sum_{t'+t''=t}
\pi_i^{t't''+a't''}
q_i^{a't''-a''t'}
\bbinom{a'}{t'}_i\bbinom{a''}{t''}_i.
\end{equation}

We have $\bbinom{-1}{t}_i=(-1)^t\pi_i^{\binom{t+1}{2}}$ for any
$t\geq 0$, $i\in I$.

For {\em $(q,\pi)$-integers} we shall denote
\[
[n]_i=\bbinom{n}{1}_i=\frac{(\pi_iq_i)^n-q_i^{-n}}{\pi_iq_i-q_i^{-1}}\quad
\text{for } n\in\Z,
\]
\[
[n]^!_i=\prod_{s=1}^n [s]_i\quad \text{for } n\in\N,
\]
and with this notation we have
\[
\bbinom{a}{t}_i=\frac{[a]_i^!}{[t]_i^![a-t]_i^!}\quad \text{for
}0\leq t \leq a.
\]
Note that the  $(q,\pi)$-integers $[n]_i$ and the $(q,\pi)$-binomial
coefficients in general are not necessarily bar-invariant unless the
super Cartan datum is consistent; see \eqref{eq:barinvqi}.

If $a\geq 1$, then we have
\begin{equation}\label{eq:binomidf}
\sum_{t=0}^a(-1)^t\pi_i^{\binom{t}{2}}q_i^{t(a-1)}\bbinom{a}{t}_i=0
\end{equation}
which follows from \eqref{eq:binomidc} by setting $z=-1$.

If $x,y$ are two elements in a $\Qq^\pi$-algebra such that $xy=\pi_i
q_i^2 yx$, then for any $a\geq 0$, we have the quantum binomial
formula:
\begin{equation}\label{eq:quantumbinomial}
(x+y)^a=\sum_{t=0}^a q_i^{t(a-t)}\bbinom{a}{t}_iy^tx^{a-t}.
\end{equation}

\subsection{Quantum Serre relations}
 \label{subsec:Af}

For any $n\in \Z$, let the divided powers $\theta_i^{(n)}$ (in $\ff$
or $\ffpr$) be defined as $\theta_i/[n]_i^!$ if $n\geq 0$ and $0$
otherwise.

\begin{lem}\label{lem:divpowcoprod}
For any $n\in\Z$ we have
\begin{enumerate}
\item[(a)] $r(\theta_i^{(n)}) =\sum_{t+t'=n} q_i^{tt'}
\theta_i^{(t)}\otimes \theta_i^{(t')}$,
\item[(b)] $\ov{r}(\theta_i^{(n)}) =\sum_{t+t'=n} (\pi_i q_i)^{-tt'}
\theta_i^{(t)}\otimes \theta_i^{(t')}$.
\end{enumerate}
\end{lem}

\begin{proof}
By the quantum binomial formula \eqref{eq:quantumbinomial} applied to $x=1\otimes
\theta_i$ and $y=\theta_i\otimes 1$, the formula follows.
\qed\end{proof}

\begin{lem}\label{lem:divpownorm}
For any $n\geq 0$, we have
\[
(\theta_i^{(n)},\theta_i^{(n)})=\prod_{s=1}^n
\frac{\pi_i^{s-1}}{1-(\pi_iq_i^{-2})^s}
=\pi_i^nq_i^{\binom{n+1}{2}}(\pi_iq_i-q_i^{-1})^{-n}
([n]^!_i)^{-1}.
\]
\end{lem}

\begin{proof}
We prove by induction on $n$. The lemma is true by definition for $n=0,1$.
For general $n$, it follows by Lemma~\ref{lem:divpowcoprod}(a) that
\begin{align*}
\parens{\theta_i^{(n)},\theta_i^{(n)}}&=[n]_i^{-1}
\parens{\theta_i^{(n-1)}\otimes\theta_i,r\parens{\theta_i^{(n)}}}
 \\
&=[n]_i^{-1}\parens{\theta_i^{(n-1)}\otimes\theta_i,\sum_{t+t'=n} q_i^{tt'}
 \theta_i^{(t)}\otimes \theta_i^{(t')}}
  \\
&=[n]_i^{-1}\parens{\theta_i^{(n-1)}\otimes\theta_i, q_i^{n-1}\theta_i^{(n-1)}\otimes\theta_i}\\
&=q_i^{n-1}[n]_i^{-1}(\theta_i,\theta_i)\parens{\theta_i^{(n-1)},\theta_i^{(n-1)}}.
 \\
\intertext{Hence by the induction hypothesis, we have}
\parens{\theta_i^{(n)},\theta_i^{(n)}}
&=q_i^{n-1}[n]_i^{-1}(1-\pi_iq_i^{-2})^{-1}
\pi_i^{n-1}q_i^{\binom{n}{2}}(\pi_iq_i-q_i^{-1})^{-n+1} ([n-1]^!_i)^{-1}\\
&=\pi_i^{n}q_i^{\binom{n+1}{2}}(\pi_iq_i-q_i^{-1})^{-n} ([n]^!_i)^{-1}.
\end{align*}
The lemma is proved.
\qed\end{proof}

\begin{prop}
[Quantum Serre relation] \label{prop:quantserre}
The generators
$\theta_i$ of $\ff$ satisfy the relations
\[
\sum_{n+n'=1-\ang{i,j'}}(-1)^{n'}\pi_i^{n'p(j)+\binom{n'}{2}}
\theta_i^{(n)}\theta_j\theta_i^{(n')}=0
\]
for any $i\neq j$ in $I$.
\end{prop}

Proposition~\ref{prop:quantserre} appeared as \cite[Theorem~3.8]{HW}.
We shall give a new and simpler proof of Proposition~\ref{prop:quantserre} 
below after some preparation.

\begin{lem}
Let $N\in \N$ and $a,a'\in \N$ with $N=a+a'$. Let $i,j,k\in I$ be
pairwise distinct. Then
\begin{enumerate}
\item[(a)] $r_k(\theta_{i}^{(a)}\theta_j\theta_i^{(a')})
 =0$,
\item[(b)] $r_j(\theta_{i}^{(a)}\theta_j\theta_i^{(a')})
=q_i^{a'\ang{i,j}}\pi_i^{a'p(j)}\bbinom{N}{a'}_i\theta_i^{(N)}$,
\item[(c)] $r_i(\theta_{i}^{(a)}\theta_j\theta_i^{(a')})
=q_i^{a'+(N+\ang{i,j}-1)}\pi_i^{a'+p(j)}\theta_i^{(a-1)}\theta_j\theta_i^{(a')}
+q_i^{a'-1}\theta_i^{(a)}\theta_j\theta_i^{(a'-1)}$.
\end{enumerate}
\end{lem}

\begin{proof}
Part (a) is clear from definitions. By \eqref{eq:derivcoprod} and
Lemma~ \ref{lem:divpowcoprod}(a) we have
\[
r_{i'}(\theta_{j'}^{(a)})=\delta_{i',j'}q_{i'}^{a-1}\theta_{i'}^{(a-1)}.
\]
Parts (b) and (c) follow from this and noting
\[
r_i(cba)=cbr_i(a)+\pi^{p(i)p(a)}q^{i\cdot|a|}cr_i(b)a
+\pi^{p(i)p(a)+p(i)p(b)}q^{i\cdot|a|+i\cdot|b|}r_i(c)ba.
\]
The lemma is proved.
\qed\end{proof}

\subsubsection{Proof of Proposition~\ref{prop:quantserre}}
\label{subsec:proofSerre}

Let $N=1-\ang{i,j'}$.
By the previous lemma, we have
\begin{align*}
r_k(\sum_{n+n'=N}(-1)^{n'}\pi_i^{n'p(j)+\binom{n'}{2}}\theta_i^{(n)}\theta_j\theta_i^{(n')})&=0,
\quad \text{ for } k\neq i,j.
\end{align*}
In addition, we have
\begin{align*}
r_j & \parens{\sum_{n+n'=N} (-1)^{n'} \pi_i^{n'p(j)+\binom{n'}{2}}
\theta_i^{(n)}\theta_j\theta_i^{(n')}} \\
&=\sum_{n+n'=N}(-1)^{n'}\pi_i^{n'p(j)
+\binom{n'}{2}}q_i^{n'\ang{i,j}}\pi_i^{n'p(j)}\bbinom{N}{n'}_i\theta_i^{(N)}\\
&=\theta_i^{(N)}\sum_{t=0}^N (-1)^{t}\pi_i^{\binom{t}{2}}(q_i)^{t(1-N)}\bbinom{N}{t}_i.
\end{align*}
By Condition \ref{subsec:Cartan}(e), $1-N\in 2\Z$ if $i$ is odd, so
in any case, the right-hand side of the last equation is
\begin{align*}
&=\theta_i^{(N)}\sum_{t=0}^N (-1)^{t}\pi_i^{\binom{t}{2}}(\pi_i q_i^{-1})^{t(N-1)}\bbinom{N}{t}_i =0,
\end{align*}
 where the last equality follows from \eqref{eq:binomidf}. Finally,
\begin{align*}
r_i &\parens{\sum_{n+n'=N}(-1)^{n'}\pi_i^{n'p(j)+\binom{n'}{2}}\theta_i^{(n)}
\theta_j\theta_i^{(n')}}
 \\
&=\sum_{n+n'=N} (-1)^{n'}\pi_i^{n'p(j)+\binom{n'}{2}}q_i^{n'}\pi_i^{n'+p(j)}
\theta_i^{(n-1)}\theta_j\theta_i^{(n')}\\
&\hspace{.5 in}+\sum_{n+n'=N} (-1)^{n'} 
\pi_i^{n'p(j)+\binom{n'}{2}}q_i^{n'-1}\theta_i^{(n)}
\theta_j\theta_i^{(n'-1)} 
 \\
&=\sum_{t=0}^{N-1} (-1)^{t}\pi_i^{tp(j)+p(j)
 +\binom{t+1}{2}} q_i^{t}\theta_i^{(N-1-t)}\theta_j\theta_i^{(t)}\\
&\hspace{.5 in}-\sum_{t=0}^{N-1} (-1)^{t}\pi_i^{(t+1)p(j)+\binom{t+1}{2}}q_i^{t}\theta_i^{(N-1-t)}\theta_j\theta_i^{(t)}\\
 &=0.
\end{align*}

Now Proposition~\ref{prop:quantserre} follows by Lemma \ref{lem:derivkill}.
\qed

Note that the bar map $\ov{\phantom{x}}$ on $\ff$ may not be
well-defined when the datum is not consistent. For example, consider
the case $(I,\cdot)$ has $i,j\in I_\zero$ with $i\cdot j=-1$, hence
$d_i=d_j=1$. Then the calculations above hold; that is,
$s(\theta_i,\theta_j)
:=\theta_i^{(2)}\theta_j-\theta_i\theta_j\theta_i+\theta_j\theta_i^{(2)}=0$;
however, since $\ov{[2]_i}=\pi[2]_i$, it is easy to see that
$\ov{s(\theta_i,\theta_j)}\notin \mathcal{I}$.


Let $_\curlyA\ff$ be $\curlyA^\pi$-subalgebra of $\ff$ generated by
the elements $\theta_i^{(s)}$ for various $i\in I$ and $s\in\Z$.
Since the generators $\theta_i^{(s)}$ are homogeneous, we have
$_\curlyA\ff=\bigoplus_\nu\,{}_\curlyA\ff_\nu$ where $\nu$ runs over
$\N[I]$ and $_\curlyA\ff_\nu= {}_\curlyA\ff \cap\ff_\nu$.

\section{The quantum covering and super groups}   \label{sec:algebraU}

In this section we give the definition of the quantum covering
group $\UU$ as a Hopf superalgebra, which specializes at $\pi=-1$
to a new variant of a quantum supergroup. We show that $\UU$ admits a
triangular decomposition $\UU =\UU^-\UU^0\UU^+$ with
positive/negative parts isomorphic to the algebra $\bf f$. The
novelty here is that $\UU^0$ contains some new generators $J_i (i\in
I)$ which allow us to construct integrable modules in full
generality.

\subsection{The algebras $\UUpr$  and $\UU$}

\label{sec:Udef} Assume that a root datum $(Y,X, \ang{\,,\,})$
of type $(I,\cdot)$ is given.
Consider the associative $\Qq^\pi$-superalgebra $\UUpr$ (with $1$)
defined by the generators
\[\
E_i\quad(i\in I),\quad F_i\quad (i\in I), \quad J_{\mu}\quad (\mu\in
Y),\quad K_\mu\quad(\mu\in Y),
 \]
where the parity is given by $p(E_i)=p(F_i)=p(i)$ and
$p(K_\mu)=p(J_\mu)=0$, subject to the relations (a)-(f) below for
all $i, j \in I, \mu, \mu'\in Y$:
\[
\tag{a} K_0=1,\quad K_\mu K_{\mu'}=K_{\mu+\mu'},
\]
\[
\tag{b} J_{2\mu}=1, \quad J_\mu J_{\mu'}=J_{\mu+\mu'},
\]
\[
\tag{c} J_\mu K_{\mu'}=K_{\mu'}J_{\mu},
\]
\[
\tag{d} K_\mu E_i=q^{\ang{\mu,i'}}E_iK_{\mu}, \quad
J_{\mu}E_i=\pi^{\ang{\mu,i'}} E_iJ_{\mu},
\]
\[
\tag{e} \; K_\mu F_i=q^{-\ang{\mu,i'}}F_iK_{\mu}, \quad J_{\mu}F_i=
\pi^{-\ang{\mu,i'}} F_iJ_{\mu},
\]
\[
\tag{f} E_iF_j-\pi^{p(i)p(j)}
F_jE_i=\delta_{i,j}\frac{\wtd{J}_{i}\wtd{K}_i-\wtd{K}_{-i}}{\pi_iq_i-q_i^{-1}},
\]
where for any element $\nu=\sum_i \nu_i i\in \Z[I]$ we have set
$\wtd{K}_\nu=\prod_i K_{d_i\nu_i i}$, $\wtd{J}_\nu=\prod_i
J_{d_i\nu_i i}$. In particular, $\wtd{K}_i=K_{d_i i}$,
$\wtd{J}_i=J_{d_i i}$. (Under Condition~ \ref{subsec:Cartan}(e),
$\wtd{J}_i=1$ for $i\in I_\zero$ while $\wtd{J}_i=J_i$ for $i\in
I_\one$.)

We also consider the associative $\Qq^\pi$-algebra $\UU$ (with $1$)
defined by the generators
\[\
E_i\quad(i\in I),\quad F_i\quad (i\in I), \quad J_{\mu}\quad
(\mu\in Y),\quad K_\mu\quad(\mu\in Y)
\]
and the relations (a)-(f) above, together with the additional
relations
\[
\tag{g}\text{for any } f(\theta_i:i\in I)\in \mathcal{I} \text{,
}f(E_i:i\in I)=f(F_i:i\in I)=0.
\]
The algebra $\UU$ will be called the {\em quantum covering group} of
type $(I,\cdot)$.

{From} (g), we see that there are well-defined algebra homomorphisms
$\ff\rightarrow \UU$, $x\mapsto x^+$  (with image denoted
by $\UU^+$) and $\ff\rightarrow \UU$, $x\mapsto x^-$
(with image denoted by $\UU^-$)  such that
$E_i=\theta_i^+$ and $F_i=\theta_i^-$ for all $i\in I$.
Clearly, there are well defined algebra homomorphisms $\ffpr\rightarrow\UUpr$
with the aforementioned properties.

(In terms of standard notations used in some other quantum group
literature, it is understood that $K_\mu=q^\mu$ and $K_i=q^{h_i}$.
It is instructive to see our new generators $J$'s can be understood
in the same vein as
$J_\mu=\pi^\mu$ and $J_i=\pi^{h_i}$.)

For any $p\geq 0$, we set $E_i^{(p)}=(\theta_i^{(p)})^+$ and
$F_i^{(p)}=(\theta_i^{(p)})^-$.

\begin{example}\label{example:rankoneqg}
In the case $I=I_\one=\set{I}$, we can identify $Y=X=\Z$ with $i=1\in Y$,
$i'=2\in X$, and $\ang{\mu,\lambda}=\mu\lambda$. Then $\UU$ is the
$\Qq^\pi$-algebra generated by $E$,$F$,$K$,$J$ such that
\[JK=KJ,\quad JE=EJ,\quad JF=FJ,\quad J^2=1,\]
\[KEK=q^2E,\quad KFK=q^{-2}F,\]
\[EF-\pi FE=\frac{JK-K^{-1}}{\pi q-q^{-1}}.\]
Note that the quotient algebras $\UU/((J\pm 1)\UU)$ are isomorphic to the 
two variants of the quantum group $\UU_q(\osp(1|2))$ defined
in \cite{CW}. 
\end{example}

\subsection{Properties of $\UU$}

By inspection, 
there is a unique algebra automorphism (of order 4) $\omega:
\UUpr\rightarrow \UUpr$ such that
\[\omega(E_i)=\pi_i \wtd{J}_{i}F_i,\quad \omega(F_i)= E_i,
\quad \omega(K_\mu)=K_{-\mu}, \quad \omega(J_{\mu})=J_{\mu}\]
 for $i\in I$, $\mu\in Y$. 
We have $\omega(x^+)=\pi_{|x|}\tilde{J}_{|x|} x^-$
and $\omega(x^-)=x^+$ for all $x\in \ff$,
and thus the same formula defines
a unique algebra automorphism $\omega:\UU\rightarrow \UU$.

Similarly, there is a unique 
isomorphism of $\Qq^\pi$-vector spaces $\sigma:
\UUpr\rightarrow \UUpr$ such that
\[\sigma(E_i)=E_i,\quad \sigma(F_i)=\pi_i \tilde{J}_{i} F_i,
\quad \sigma(K_\mu)=K_{-\mu}, \quad \sigma(J_{\mu})=J_{\mu}\] for
$i\in I$, $\mu\in Y$ such that $\sigma(uu')=\sigma(u')\sigma(u)$ for
$u,u'\in\UU$. We have
\begin{equation} \label{eq:sigmadef}
\sigma(x^+)= \sigma(x)^+, \qquad
\sigma(x^-)=\pi_{|x|}\tilde{J}_{|x|} \sigma(x)^-, \quad \forall x\in
\ff.
\end{equation}
Again, this implies that the same formula defines a unique
algebra automorphism $\sigma:\UU\rightarrow \UU$.
Note that $\sigma$ on $\UU^+$ matches exactly $\sigma$ on $\bf f$,
but $\sigma$ on $\UU^-$ looks quite different from  $\sigma$ on $\bf
f$ (in contrast to the quantum group setting \cite{Lu}).

\begin{lem}[Comultiplication]
There is a unique algebra homomorphism $\Delta:\UUpr\rightarrow
\UUpr\otimes \UUpr$ (resp. $\Delta:\UU\rightarrow \UU\otimes \UU$)
where $\UUpr\otimes \UUpr$ (resp. $\UU\otimes \UU$) is regarded as a
superalgebra in the standard way, defined by
\begin{align*}
\Delta(E_i) &= E_i\otimes 1 + \wtd{J}_i\wtd{K}_i\otimes E_i\quad
(i\in I),
 \\
\Delta(F_i) &= F_i\otimes \wtd{K}_{-i} + 1\otimes F_i\quad (i\in
I),
 \\
\Delta(K_\mu) &=K_\mu\otimes K_\mu\quad (\mu\in Y),
 \\
\Delta(J_{\mu}) &=J_{\mu}\otimes J_{\mu}\quad (\mu\in Y).
\end{align*}
\end{lem}

\begin{proof}
The relations \ref{sec:Udef} (a)-(c) are trivial to verify.
For the relation (d), we have
\[
\Delta(E_i)\Delta(F_j)=E_iF_j\otimes
\wtd{K}_{-j}+\wtd{J}_i\wtd{K}_i\otimes E_iF_i + E_i\otimes
F_j+\pi^{p(i)p(j)}\wtd{J}_i\wtd{K}_iF_j\otimes E_i\wtd{K}_{-j},
\]
\[
\Delta(F_j)\Delta(E_i)=F_iE_j\otimes
\wtd{K}_{-j}+\wtd{J}_i\wtd{K}_i\otimes F_jE_i +
\pi^{p(i)p(j)}E_i\otimes F_j+F_j\wtd{J}_i\wtd{K}_i\otimes
\wtd{K}_{-j}E_i.
\]
So using the fact that $F_j\wtd{K}_i\otimes
\wtd{K}_{-j}E_i=\wtd{K}_iF_j\otimes E_i\wtd{K}_{-j}$, we have
\begin{align*}
\Delta(E_i) & \Delta(F_j)-\pi^{p(i)p(j)}\Delta(F_j)\Delta(E_i)\\
 &=(E_iF_j-\pi^{p(i)p(j)}F_jE_i)\otimes
  \wtd{K}_{-j}+\wtd{J}_i\wtd{K}_i\otimes (E_iF_j-\pi^{p(i)p(j)}F_jE_i)\\
&=\delta_{i,j}\parens{\frac{\wtd{J}_i\wtd{K}_i-\wtd{K}_{-i}}{\pi_iq_i-q_i^{-1}}}
\otimes\wtd{K}_{-j}
+\wtd{J}_i\wtd{K}_{i}\otimes
 \parens{\delta_{i,j}\frac{\wtd{J}_i\wtd{K}_i-\wtd{K}_{-i}}{\pi_iq_i-q_i^{-1}}}\\
&=\delta_{i,j}\frac{\Delta(\wtd{J}_i)\Delta(\wtd{K}_i)
 -\Delta(\wtd{K}_{-i})}{\pi_iq_i-q_i^{-1}}.
\end{align*}

Finally, define maps $j^\pm:\ffpr\otimes\ffpr\rightarrow
\UUpr\otimes \UUpr$ given by
\[
j^+(x\otimes y)=x^+ \wtd{J}_{|y|}\wtd{K}_{|y|}\otimes y^+, \qquad
j^-(x\otimes y)=x^-\otimes\wtd{K}_{-|x|}y^-.
\]
Then by construction, these maps are algebra homomorphisms, and
satisfy \[j^+r(x)=\Delta(x^+),\quad j^-\ov{r}(x)=\Delta(x^-).\] Since
$r,\ov{r}$ factor through $\ff$, so do $j^+r$ and $j^-\ov{r}$
implying that \[f(\Delta(E_i))=f(\Delta(F_i))=0\] for all
$f(\theta_i:i\in I)\in \mathcal{I}$.
\qed\end{proof}

The previous proof shows that $j^+r(x)=\Delta(x^+)$ and $j^-
\ov{r}(x)=\Delta(x^-)$, so in particular we have
\begin{align*}
\Delta(x^+) &=\sum x_1^+ \wtd{J}_{|x_2|}\wtd{K}_{|x_2|}\otimes
x_2^+,
 \\
\Delta(x^-) &=\sum \pi^{p(x_1)p(x_2)}(\pi q)^{-|x_1|\cdot
|x_2|}x_2^- \otimes\wtd{K}_{-|x_2|} x_1^-,
\end{align*}
for $r(x)=\sum x_1\otimes x_2$. In particular, this yields the
formulas
\begin{align*}
\Delta(E_i^{(p)})
 &=\sum_{p'+p''=p}q_i^{p'p''}\wtd{J}_i^{p''}E_i^{(p')}\wtd{K}_i^{p''}\otimes E_i^{(p'')}, \\
\Delta(F_i^{(p)})
 &=\sum_{p'+p''=p}(\pi_i q_i)^{-p'p''}F_i^{(p')}\otimes\wtd{K}_i^{-p'}
 F_i^{(p'')}.
\end{align*}

\begin{prop}\label{prop:EFcommutation}
For $x\in \ffpr$ and $i\in I$, we have (in $\UUpr$)
\vspace{.1in}
\begin{enumerate}
\item[(a)] $\displaystyle
x^+F_i-\pi_i^{p(x)}F_ix^+=\frac{\ri(x)^+\wtd{J}_i\wtd{K}_i-\wtd{K}_{-i}\,\,
\pi_i^{p(x)-p(i)}\, \ir(x)^+}{\pi_iq_i-q_i^{-1}},
$
\vspace{.1in}
\item[(b)] $\displaystyle
E_ix^--\pi_i^{p(x)}x^-E_i=\frac{\wtd{J}_i\wtd{K}_{i}\,\,
\ir(x)^- -\pi_i^{p(x)-p(i)}\ri(x)^-\wtd{K}_{-i}}{\pi_iq_i-q_i^{-1}}.
$
\end{enumerate}
\end{prop}
\begin{proof}

Assume that (a) is known for $x'$ and $x''$; we shall show it holds
for $x=x'x''$. Let $y'=(x')^+$, $_i y'=\ir(x')^+$ and similarly for
$\ri$, $x''$ and $x$.
\begin{align*}
 yF_i&=\pi_i^{p(x'')}y'F_iy''+\frac{y' y''_i \wtd{J}_i
 \wtd{K}_i-y'\wtd{K}_{-i}\,\, \pi_i^{p(x'')-p(i)}\,{}_iy''}{\pi_iq_i-q_i^{-1}}\\
&=\pi_i^{p(x'x'')}F_iy+\pi_i^{p(x'')}\frac{y'_i \wtd{J}_i
 \wtd{K}_i\, y''-\wtd{K}_{-i} \pi_i^{p(x')-p(i)}\,{}_iy' y''}{\pi_iq_i-q_i^{-1}}
  \\ &
 \quad +\frac{y'\,y''_i\wtd{J}_i\wtd{K}_i-y'\wtd{K}_{-i} \pi_i^{p(x'')-p(i)}\,{}_iy''}{\pi_iq_i-q_i^{-1}}\\
&=\pi^{p(x'x'')p(i)}F_i\,y+\frac{y_i\wtd{J}_i\wtd{K}_i
 -\wtd{K}_{-i}\pi_i^{p(x)-p(i)}\, {}_iy}{\pi_iq_i-q_i^{-1}}.
\end{align*}
Since (a) holds for the generators, it holds for all $x\in\ffpr$.

If we apply $\omega^{-1}$, we obtain
\[
\pi_i\wtd{J}_ix^-E_i-\pi_i^{p(x)-p(i)}
\wtd{J}_iE_ix^-=\frac{\ri(x)^-\wtd{J}_i\wtd{K}_{-i}-\wtd{K}_{i}\,\,
\pi_i^{p(x)-p(i)}\,\ir(x)^-}{\pi_iq_i-q_i^{-1}},
\]
and multiplying both sides by $\pi_i^{p(x)-p(i)} \wtd{J}_i$
establishes (b).
\qed\end{proof}


We record the following formulas for further use.
\begin{lem}\label{L:commutation} \cite[Lemma~2.8]{CW}
For any $N,M\geq 0$ we have in $\UU$ or $\UUpr$
\begin{align*}
&E_i^{(N)}F_i^{(M)}=\sum_t \pi_i^{MN-\binom{t+1}{2}} F_i^{(M-t)}
 \bbinom{\tK_i;2t-M-N}{t}_iE_i^{(N-t)}, \\
&F_i^{(N)}E_i^{(M)}=\sum_t (-1)^t \pi_i^{(M-t)(N-t)-t^2}
 E_i^{(M-t)}\bbinom{\tK_i;M+N-(t+1)}{t}_iF_i^{(N-t)}, \\
&E_i^{(N)}F_j^{(M)}=\pi^{MNp(i)p(j)} F_j^{(M)}E_i^{(N)}\quad\text{if }i\neq j,
\end{align*}
where
\[
\bbinom{\tK_i;a}{t}_i=\prod_{s=1}^t\frac{(\pi_iq_i)^{a-s+1} \tJ_i
\tK_i-q_i^{s-a-1}\tK_{-i}}{(\pi_iq_i)^s-q_i^{-s}}.
\]
\end{lem}


The coproduct $\Delta$ is coassociative; the verification is the same
as in the non-super case.
%
%
There is a unique algebra homomorphism $e:\UU\rightarrow \Qq^\pi$
satisfying $e(E_i)=e(F_i)=0$ and $e(J_\mu)= e(K_\mu)=1$ for all $i,\mu$.


Recall the bar involution $\ov{\phantom{x}}$ on $\Qq^\pi$ from
\eqref{eq:barinvqi}. This extends to a unique homomorphism of
$\Q$-algebras $\ov{\white{x}}:\UU\rightarrow\UU$ such that
\[\ov{E_i}=E_i,\quad \ov{F_i}=F_i,\quad \ov{J_{\mu}}=J_{\mu},\quad
\ov{K_\mu}=J_{\mu}K_{-\mu},
\]
and $\ov{fx}=\ov{f}\ov{x}$ for all $f\in\Qq^\pi$ and $x\in\UU$.


Let $_\curlyA\UU^{\pm}$ be the images of $_\curlyA\ff$ defined at the end of
\S\ref{subsec:Af}. We define $_\curlyA \UU$ to be the
$\curlyA^\pi$-subalgebra of $\UU$ generated by $E_i^{(t)}$,
$F_i^{(t)}$, $\left[K_i;a\atop t\right]_i$, $J_{\mu}$ and $K_\mu$,
for all $i\in I$, $\mu\in Y$ and positive integers $a \ge t$.

\subsection{Triangular decompositions for $\UUpr$ and $\UU$}

If $M', M$ are two $\UUpr$-modules, then $M'\otimes M$ is naturally
a $\UUpr\otimes \UUpr$-module; hence by restriction to $\UUpr$ under
$\Delta$, it is a $\UUpr$-module.

\begin{lem}
Let $\lambda\in X$. There is a unique $\UUpr$-module structure on
the $\Qq^\pi$-module $\ffpr$ such that for any homogeneous
$z\in\ffpr$, and $\mu\in Y$ and any $i\in I$, we have
\[
K_\mu\cdot z=q^{\ang{\mu,\lambda-|z|}} z,\quad
J_{\mu}\cdot z=\pi^{\ang{\mu,\lambda-|z|}} z,\quad F_i\cdot
z=\theta_i z,\quad E_i\cdot 1=0.
\]
\end{lem}
\begin{proof}
The uniqueness is immediate. To prove the existence, define
\[
E_i\cdot z=\frac{-q_i^{\ang{i,\lambda}} \ri(z) + \pi_i^{p(z)-p(i)}
(\pi_i q_i)^{\ang{i, \lambda-|z|+i'}} \ir(z) }{\pi_iq_i-q_i^{-1}}
\]
Note that this is essentially the formula prescribed by Proposition
\ref{prop:EFcommutation}. A straightforward computation shows that
this, along with the desired formulas for the $F$ and $K$ actions
define a $\UUpr$-module structure on $\ffpr$.
\qed\end{proof}

We denote this $\UUpr$-module by $\mathbf M_\lambda$ (which is a free
$\Qq^\pi$-module). Similarly, to an element $\lambda\in X$, we
associate a unique $\UUpr$-module structure on $\ffpr$ such that for
any homogeneous $z\in\ffpr$, any $\mu\in Y$ and any $i\in I$ we have
\[
K_\mu\cdot z=q^{\ang{\mu,-\lambda+|z|}} z,\quad
J_{\mu}\cdot z=\pi^{\ang{\mu,-\lambda+|z|}} z,\quad E_i\cdot
z=\theta_i z,\quad F_i\cdot 1=0.
\]
We denote this $\UUpr$-module by $\mathbf M_\lambda'$ (which is again a free
$\Qq^\pi$-module). We form the $\UUpr$-module 
$\mathbf M_\lambda'\otimes \mathbf M_\lambda$; we
denote the unit element of $\ffpr=\mathbf M_\lambda$ by $1$ 
and that of $\ffpr=\mathbf M_\lambda'$
by $1'$. Thus, we have the canonical element $1'\otimes 1\in
\mathbf M_\lambda'\otimes \mathbf M_\lambda$. We emphasize that 
$\mathbf M_\lambda'\otimes \mathbf M_\lambda$ is again free as 
a $\Qq^\pi$-module.

\begin{prop}
 \label{prop:U'triang}
Let $\Uz$ be the associative $\Qq^\pi$-algebra with $1$ defined by
the generators $K_\mu$, $J_{\mu}$ ($\mu\in Y$) and the relations in
\S\ref{sec:Udef} (a),(b). Then $\Uz$ is isomorphic to the group algebra of
$Y\times (Y/2Y)$ over $\Qq^\pi$. Moreover,
\begin{enumerate}
\item[(a)]
The $\Qq^\pi$-linear map $\ffpr\otimes \Uz\otimes
\ffpr\rightarrow\UUpr$ given by $u\otimes J_\nu K_\mu\otimes
w\mapsto u^-J_{\nu} K_\mu w^+$ is an isomorphism.
\item[(b)]
The $\Qq^\pi$-linear map $\ffpr\otimes \Uz\otimes
\ffpr\rightarrow\UUpr$ given by $u\otimes J_{\nu}K_\mu\otimes
w\mapsto u^+J_\nu K_\mu w^-$ is an isomorphism.
\end{enumerate}
\end{prop}

\begin{proof}
Note that (b) follows from (a) by applying $\omega$. As a
$\Qq^\pi$-module, $\UUpr$ is spanned by words in the $E_i$, $F_i$,
$K_\mu$, and $J_{\mu}$. By using the defining relations, we can
rewrite any word as a linear combination of words where the $F_i$
come before the $J_{\mu}$ and $K_\mu$, which come before the $E_i$,
thus the given map is surjective.

To prove the map is injective, let $\lambda,\lambda'\in X$, and
consider the module $\mathbf M_{\lambda'}'\otimes \mathbf M_\lambda$ 
described before. There is a
$\Qq^\pi$-linear map 
$\phi : \UUpr\rightarrow \mathbf M_{\lambda'}'\otimes \mathbf M_\lambda$
given by
$\phi(u)=u\cdot 1'\otimes 1$. Pick a $\Qq^\pi$-basis of $\ffpr$
consisting of homogeneous elements containing $1$. Assume that in
$\UUpr$ there is some relation of the form
$\sum_{b',\mu,b}c_{b',\mu,b}b'^-J_\nu K_\mu b^+ = 0$ and let $N$ be the
largest integer such that $\height |b'|=N$ and $c_{b',\mu, b}\neq 0$
for some $\mu, b$.

Then
\[
0=\phi\Big(\sum_{b',\mu,\nu,b}c_{b',\mu,\nu,b}b'^- J_{\nu}K_\mu
b^+\Big) =\sum_{b',\mu,\nu,b}c_{b',\mu,\nu,b}\Delta(b'^-
J_{\nu}K_\mu b^+)\cdot 1\otimes 1.
\]
Now
\[
\Delta(b'^-)=\sum_{b'_1, b'_2} g'(b',b'_1,b'_2)b'^-_1
\otimes\wtd{K}_{-|b'_1|} b'^-_2,
\]
\[
\Delta(b^+)=\sum_{b_1, b_2} g(b,b_1,b_2)b_1^+
\wtd{J}_{|b_2|}\wtd{K}_{|b_2|}\otimes b_2^+,
\]
so we have
\begin{align*}
0=& \sum \pi^{p(b'_2)p(b_1)}
c_{b',\mu,\nu,b}g(b,b_1,b_2)g'(b',b'_1,b'_2)b'^-_1  \times
 \\
 &
 \qquad\qquad\qquad\times
 J_{\nu}K_\mu b_1^+ \wtd{J}_{|b_2|}\wtd{K}_{|b_2|}\cdot 1' \otimes
\wtd{K}_{-|b'_1|} b'^-_2 J_{\nu} K_\mu b_2^+\cdot 1.
\end{align*}
If $b_2\neq 1$, then $b_2^+\cdot 1=0$ so we must have $b_2=1$ and
thus $b_1=b$. Therefore the expression reduces to
\[
0=\sum \pi^{p(b'_2)p(b)}c_{b',\mu,\nu,b}g'(b',b'_1,b'_2)b'^-_1
J_{\nu}K_\mu b^+\cdot 1' \otimes \wtd{K}_{-|b'_1|} b'^-_2
J_{\nu}K_\mu \cdot 1.
\]
By the definition of the module structure, this  becomes
\[
0=\sum
\pi^{p(b'_2)p(b)}c_{b',\mu,\nu,b}g'(b',b'_1,b'_2)\pi^{\ang{\nu,
\lambda-\lambda'+|b|}}q^{\ang{\mu, \lambda-\lambda'+|b|}}b'^-_1
\cdot b \otimes \wtd{K}_{-|b'_1|} b'_2.
\]
We can now project this equality onto the summand 
$\mathbf M_{\lambda'}'\otimes
\ffpr_\nu$ where $\height\, \nu = N$. Then by construction,
$|b'_2|\leq |b|$ and $\height |b'_2|=N$. Since $c_{b',\mu,b}=0$ if
$\height |b'|>N$, we must have $|b|=|b'_2|$ and thus $b'=b'_2$,
$b'_1=1$, so
\[
\sum \pi^{p(b')p(b)}c_{b',\mu,\nu,b}\pi^{\ang{\nu,
\lambda-\lambda'+|b|}}q^{\ang{\mu, \lambda-\lambda'+|b|}} b \otimes
b'=0.
\]
It follows that
\[
\sum_{\nu,\mu} c_{b',\mu,\nu,b}\pi^{\ang{\nu,
\lambda-\lambda'+|b|}}q^{\ang{\mu, \lambda-\lambda'+|b|}}=0
\]
for all choices of $\lambda,\lambda', \mu, b$ and $b'$ with $\height
|b'|=N$. Therefore $c_{b',\mu,\nu,b}=0$ for any $b'$ with $\height
|b'|=N$, contradicting the choice of $N$.
\qed\end{proof}

\begin{cor}
\begin{enumerate}
\item[(a)]
The $\Qq^\pi$-linear map $\ff\otimes \Uz\otimes \ff\rightarrow\UU$
given by $u\otimes J_{\nu}K_\mu\otimes w\mapsto u^-J_{\nu}K_\mu w^+$
is an isomorphism.
\item[(b)]
The $\Qq^\pi$-linear map $\ff\otimes \Uz\otimes \ff\rightarrow\UU$
given by $u\otimes K_\mu\otimes w\mapsto u^+J_{\nu}K_\mu w^-$ is an
isomorphism.
\end{enumerate}
\end{cor}

\begin{proof}
Once again (b) follows from (a) by applying the involution $\omega$.
Let $J_\pm$  be the two-sided ideal of $\UUpr$ generated by
$\mathcal{I}^\pm=\set{x^\pm: x\in\mathcal I}$. Then
$\UU=\frac{\UUpr}{J_++J_-}$. Now from Proposition
\ref{prop:EFcommutation} iterated, we see that
$$
(\UUpr^+)\mathcal{I}^-\subseteq\mathcal{I}^-\Uz(\UUpr^+);
\qquad
\mathcal{I}^+(\UUpr^-)\subseteq(\UUpr^-)\Uz\mathcal{I}^+.
$$

Using the triangular decomposition of $\UUpr$, we have
$J_-=\UUpr\mathcal{I}^-\UUpr\subseteq\mathcal{I}^-\Uz(\UUpr^+)\subseteq
J_-$, hence $J_-=\mathcal{I}^-\Uz(\UUpr^+)$. Similarly,
$J_+=(\UUpr^-)\Uz\mathcal{I}^+$. Therefore,
\[
\UU=\frac{\UUpr^-\otimes \Uz \otimes \UUpr^+}{\UUpr^-\otimes
\Uz\otimes \mathcal{I}^++\mathcal{I}^-\otimes \Uz\otimes \UUpr^+} =
\frac{\UUpr^-}{\mathcal{I}^-}\otimes \Uz\otimes
\frac{\UUpr^+}{\mathcal{I}^+},
\]
from which (a) follows.
\qed\end{proof}

\begin{cor}
The maps $\vphantom{x}^\pm:\ff\rightarrow \UU^\pm$, $x\mapsto
x^{\pm}$, are $\Qq^\pi$-algebra isomorphisms, and $\Uz\rightarrow
\UU$ is a $\Qq^\pi$-algebra embedding.
\end{cor}

For $\nu\in \N[I]$, we shall denote the image $\ff_\nu^\pm$ by $\UU^\pm_\nu$.

\begin{prop}  \label{prop:x=0}
Let $x\in\ff_\nu$ where $\nu\in \N[I]$ is nonzero.
\begin{enumerate}
\item[(a)] If $x^+F_i=\pi_i^{p(x)} F_ix^+$ for all $i\in I$ then $x=0$.
\item[(b)] If $x^-E_i=\pi_i^{p(x)}E_ix^-$ for all $i\in I$ then $x=0$.
\end{enumerate}
\end{prop}

\begin{proof}
It follows from Proposition \ref{prop:EFcommutation} and the linear
independence of $\ri(x)^+ \wtd{J}_i\wtd{K}_i$ (respectively, the linear
independence of
$\wtd{J}_i\wtd{K}_{-i}\,\ir(x)^+$) that $\ri(x)^+=\ir(x)^+=0$ for all $i$.
Hence $x=0$ by Lemma~\ref{lem:derivkill}.
\qed\end{proof}

\subsection{Antipode}\label{subsec:antipode}

For $\nu\in \N[I]$, write $\nu =\sum_i\nu_i i$ and 
$\nu=\sum_{a=1}^{\height \nu} i_{a}$ for $i_a\in I$. Then we set
\begin{align*}
c(\nu)&=\nu\cdot\nu / 2 - \sum_i \nu_i i\cdot i/2\in \Z,
  \\
e(\nu)&=\sum_{a<b} p(i_a)p(i_b) \in \Z.
\end{align*}

\begin{lem}
Let $\nu\in \N[I]$.
\begin{enumerate}
\item[(a)]
There is a unique $\Qq^\pi$-linear map
$S:\UU\rightarrow \UU$ such that
\[
S(E_i)=-\wtd{J}_{-i}\wtd{K}_{-i}E_i,\ \ 
S(F_i)=-F_i\wtd{K}_{i},\ \  S(K_\mu)=K_{-\mu}, \ \ 
S(J_{\nu})=J_{-\nu},
\]
and $S(xy)=\pi^{p(x)p(y)}S(y)S(x)$ for all $x,y\in\UU$.

\item[(b)] For any $x\in \ff_\nu$, we have
\begin{align*}
S(x^+)&=(-1)^{\height \nu}\pi^{e(\nu)} (\pi
q)^{c(\nu)}\wtd{J}_{-\nu}\wtd{K}_{-\nu} \sigma(x)^+,\\
S(x^-)&=(-1)^{\height \nu}\pi^{e(\nu)}
q^{-c(\nu)}\sigma(x)^-\wtd{K}_{\nu}.
\end{align*}

\item[(c)]
There is a unique $\Qq^\pi$-linear map
$S':\UU\rightarrow \UU$ such that
\[
S'(E_i)=-E_i\wtd{J}_{-i}\wtd{K}_{-i},\ \ 
S'(F_i)=-\wtd{K}_{i}F_i,\ \  S'(K_\mu)=K_{-\mu},\ \ 
S(J_{\nu})=J_{-\nu},
\]
and $S'(xy)=\pi^{p(x)p(y)}S'(y)S'(x)$ for all $x,y\in\UU$.

\item[(d)]
For any $x\in \ff_\nu$, we have
\begin{align*}
S'(x^+)&=(-1)^{\height \nu}\pi^{e(\nu)} (\pi q)^{-c(\nu)}
\sigma(x)^+\wtd{J}_{-\nu}\wtd{K}_{-\nu},\\ 
S'(x^-)&=(-1)^{\height
\nu}\pi^{e(\nu)} q^{c(\nu)}\wtd{K}_{\nu}\sigma(x)^-.
\end{align*}

\item[(e)] We have $SS'=S'S=1$.

\item[(f)]
If $x\in \ff_\nu$, then $S(x^+)=(\pi q)^{-f(\nu)} S'(x^+)$ and
$S(x^-)=q^{f(\nu)} S'(x^-)$ where $f(\nu)=\sum_i \nu_i i\cdot i$.
\end{enumerate}
\end{lem}

The map $S$ (resp. $S'$) is called the antipode (resp. the skew-antipode) of $\UU$.
Note that
\begin{align*}
S(E_i^{(n)}) &=(-1)^n(\pi_i
q_i^2)^{\binom{n}{2}}\wtd{J}_{-ni}\wtd{K}_{-ni}E_i^{(n)},
 \\
S'(E_i^{(n)}) &=(-1)^n(\pi_i
q_i^2)^{-\binom{n}{2}}E_i^{(n)}\wtd{J}_{-ni}\wtd{K}_{-ni},
 \\
S(F_i^{(n)}) &=(-1)^n(\pi_i
q_i^2)^{-\binom{n}{2}}F_i^{(n)}\wtd{K}_{ni},
 \\
S'(F_i^{(n)}) &=(-1)^n(\pi_i
q_i^2)^{\binom{n}{2}}\wtd{K}_{ni}F_i^{(n)}.
\end{align*}

\subsection{Specializations of $\UU$ at $\pi=\pm 1$}
\label{subsec:specializations}

The {\em specialization at $\pi=1$} (respectively, at {\em
$\pi=-1$}) of a $\Qq^\pi$-algebra $R$ is understood as
$\Qq\otimes_{\Qq^\pi}R$, where $\Qq$ is the $\Qq^\pi$-module with
$\pi$ acting as $1$ (respectively, as $-1$).

Let $\mathcal{J}$ be the (2-sided) ideal of $\UU$ generated by $\{J_\mu-1|\mu\in
Y\}$.

The specialization at $\pi=-1$ of the algebra $\UU/\mathcal{J}$ is
naturally identified with a quantum group associated to the Cartan
datum $(I,\cdot)$ (cf. \cite{Lu}). The specialization at $\pi=1$ of
the algebra $\UU$, denoted by $\UU|_{\pi=1}$, is a variant of this quantum group, with some
extra (harmless) central elements $J_\mu$. Specialization at $\pi=1$
for the rest of the paper essentially reduces our results to those of Lusztig
\cite{Lu}.

The specialization at $\pi=1$ of the superalgebra $\UU/\mathcal{J}$
is identified with a quantum supergroup associated to the super
Cartan datum $(I,\cdot)$ considered in the literature; cf. \cite{Ya,
BKM}. The specialization at $\pi=-1$ of $\UU$, denoted by $\UU|_{\pi=-1}$,  will also be referred
to as a {\em quantum supergroup} of type $(I,\cdot)$, and the extra
generators $J_i$ allow us to formulate integrable modules $V(\la)$
for {\em all} $\la \in X^+$, which was not possible before.

{\bf All constructions and results in the remainder of this paper
clearly afford specializations at $\pi=-1$, which provide new
constructions and new results for quantum supergroups and their
representations. }

\subsection{The categories $\curlyC$ and $\catO$}
 \label{subsec:catCO}

 In the remainder of this paper, by a representation of the algebra $\UU$
 we mean a  $\Qq^\pi$-module on which $\UU$ acts.
 Note we have a direct sum decomposition of the $\Qq^\pi$-module
 $\Qq^\pi =(\pi+1)\Qq\oplus  (\pi-1)\Qq$ , where $\pi$ acts as $1$
on $(\pi+1)\Qq$ and   as $-1$
on $ (\pi-1)\Qq$.

We define the category $\curlyC$ (of weight $\UU$-modules) as follows. An object of $\curlyC$
is a $\Z_2$-graded $\UU$-module $M=M_\zero \oplus M_\one$,
compatible with the $\Z_2$-grading on $\UU$, with a given weight
space decomposition
\[
M=\bigoplus_{\lambda\in X} M^\lambda,
 \qquad M^\lambda=\set{m\in M \mid
K_\mu m=q^{\ang{\mu, \lambda}} m, J_\mu m=\pi^{\ang{\mu, \lambda}}
m, \forall \mu \in Y},
\]
such that $M^\lambda=M_\zero^\lambda\oplus M_\one^\lambda$ where
$M_\zero^\lambda=M^\lambda\cap M_\zero$ and
$M_\one^\lambda=M^\lambda\cap M_\one$. The $\Z_2$-graded structure
is only particularly relevant to tensor products, and will generally
be suppressed when irrelevant.
%
We have the following $\Qq^\pi$-module decomposition for each weight space:
$M^\la = (\pi+1) M^\la\oplus  (\pi-1) M^\la$; accordingly, we have
$M=M_+ \oplus M_-$ as $\UU$-modules, where
$M_\pm :=\oplus_{\la\in X}  (\pi  \pm 1) M^\la$ is an $\UU$-module on which $\pi$ acts as $\pm 1$,
i.e. a $\UU|_{\pi=\pm 1}$-module.
Hence the category $\curlyC$ decomposes into a direct sum
$\curlyC = \curlyC_+ \oplus \curlyC_-$, where $\curlyC_\pm$ can be identified with categories
of weight modules over the specializations $\UU|_{\pi=\pm 1}$.

\begin{lem}
A simple $\UU$-module is  a simple module of  either $\UU|_{\pi=1}$ or  $\UU|_{\pi=-1}$.
\end{lem}


Let $M\in \curlyC$ and let $m\in M^\lambda$. The formulas below
follow from Lemma~\ref{L:commutation}.
\begin{enumerate}
\item[(a)]
$E_i^{(N)}F_i^{(M)}m=\sum_{t} \pi_i^{MN-\binom{t+1}{2}}
\bbinom{N-M+\ang{i,\lambda}}{t}_i F_i^{(M-t)}E_i^{(N-t)}m$;

\item[(b)]
$F_i^{(M)}E_i^{(N)}m=\sum_{t} \pi_i^{(M-t)(N-t)-t^2}
\bbinom{M-N-\ang{i,\lambda}}{t}_i E_i^{(N-t)}F_i^{(M-t)}m$;

\item[(c)] $F_i^{(M)}E_j^{(N)}m=E_j^{(N)}F_i^{(M)}m$,\; for $i\neq
j$;

\item[(d)] $\;\ \bbinom{\wtd{K}_i;a}{t}_i m =\bbinom{\langle
i,\la \rangle +a}{t}_i m.$
\end{enumerate}


Note that $\UU\otimes \UU$ is a superalgebra with multiplication
$(a\otimes b)(c\otimes d)=\pi^{p(b)p(c)}ac\otimes bd$. A tensor
product of $\UU$-modules  $M\otimes N$ is naturally a $\UU\otimes
\UU$-module with the obvious diagonal grading under the action
$(x\otimes y)(m\otimes n)=\pi^{p(y)p(m)}xm\otimes yn$.

The tensor product of modules is naturally a $\UU$-module under the
coproduct action. Moreover, $\curlyC$ is closed under tensor
products. Note that for $a\in \Z_{>0}$, $M',M''\in\curlyC$, $m'\in
M'^{\lambda'}$ and $m''\in M''^{\lambda''}$, we have
\begin{align*}
E_i^{(a)}(m'\otimes m'')
 &=\sum_{a'+a''=a} \pi_i^{a''p(m')+a''\langle i, \lambda'\rangle}
q_i^{a'a''+a''\ang{i,\lambda'}} E_i^{(a')}m'\otimes E_i^{(a'')}m'',
\\
F_i^{(a)}(m'\otimes m'')
 &=\sum_{a'+a''=a}
\pi_i^{a''p(m')+a'a''}q_i^{-a'a''-a'\ang{i,\lambda''}} F_i^{(a')}m'\otimes
F_i^{(a'')}m''.
\end{align*}


To any $M\in\curlyC$, we can define a new $\UU$-module structure via
$u\cdot m=\omega(u)m$; we denote this module by $^\omega M$. By
definition, note that $^\omega M^{\lambda}=M^{-\lambda}$.


Let $\lambda\in X$. Then there is a unique $\UU$-module structure on
$\ff$ such that for any $y\in\ff$, $\mu\in Y$ and $i\in I$ we have
$K_\mu y=q^{\ang{\mu, \lambda-|y|}} y$, $J_\nu y=\pi^{\ang{\nu,
\lambda-|y|}} y$, $F_i y = \theta_i y$, and $E_i 1=0$. As in the
non-super case, this follows readily from the triangular
decomposition. This module will be called a Verma module and denoted
by $M(\lambda)$. The parity grading on $\ff$ induces a parity
grading on $M(\lambda)$ where $p(1)=0$.
As before, we have a $\UU$-module decomposition $M(\la) =M(\la)_+ \oplus M(\la)_-$,
where $M(\la)_\pm$ can be identified as the Verma module of  $\UU|_{\pi=\pm1}$ (which is a $\Qq$-vector space).

For any $M\in \curlyC$ and an element $m\in M^\lambda$ such that
$E_i m=0$ for all $i$, there is a unique $\UU$-homomorphism
$M(\lambda)\rightarrow M$ via $1\mapsto m$. This can be proved as in
\cite[3.4.6]{Lu} using now Lemma~\ref{L:commutation}.


Let $\catO$ be the full subcategory of $\curlyC$ such that for any
$M$ in $\catO$ and $m\in M$, there exists an $n\geq 0$ such that $x^+m=0$ for all
$x\in \ff_\nu$ with $\height \nu\geq n$. Note that $M(\lambda)$ and
its quotient $\UU$-modules belong to $\catO$.

\subsection{Category $\curlyC_{\mathrm{int}}$ of integrable modules}
 \label{subset:intO}

An object $M\in \curlyC$ is said to be {\em integrable} if for any
$m\in M$ and any $i\in I$, there exists $n_0\geq 1$ such that
$E_i^{(n)}m=F_i^{(n)}m=0$ for all $n\geq n_0$. Let
$\curlyC_{\mathrm{int}}$ be the full subcategory of $\curlyC$ whose
objects are the integrable $\UU$-modules.

For $M,M',M''\in\curlyC_{\mathrm{int}}$, we have $^\omega M,
M'\otimes M''\in \curlyC_{\mathrm{int}}$. The proof of the following
lemma proceeds as in the non-super case; see \cite[Lemma~3.5.3]{Lu}.

\begin{lem}
For $(a_i),(b_i)\in \N^I$ and $\lambda\in X$, let $M$ be the
quotient of $\UU$ by the left ideal generated by the elements
$F_i^{a_i+1}$, $E_i^{b_i+1}$, $K_\mu-q^{\ang{\mu,\lambda}}$ with
$\mu\in Y$, and $J_\nu-\pi^{\ang{\nu,\lambda}}$ with $\nu\in Y$. Then
$M$ is an integrable $\UU$-module.
\end{lem}

The proof of the following proposition proceeds as in the non-super
case; see \cite[Proposition~3.5.4 and 23.3.11]{Lu}.

\begin{prop}
If $u\in\UU$ such that $u$ acts as zero on every integrable module, then $u=0$.
\end{prop}

\begin{prop}
Let $\lambda\in X^+$.
\begin{enumerate}
\item[(a)]
Let $\curlyT$ be the left ideal of $\ff$ generated by the elements
$\theta_i^{\ang{i,\lambda}+1}$ for all $i\in I$. Then $\curlyT$ is a
$\UU$-submodule of the Verma module $M(\lambda)$.

\item[(b)] 
The quotient $\UU$-module $V(\lambda) :=M(\lambda)/\curlyT$ is integrable.
\end{enumerate}
\end{prop}
The proof is as in the non-super case \cite[Proposition~3.5.6]{Lu}.
As usual $V(\la) =V(\la)_+ \oplus V(\la)_-$, and $\curlyT =\curlyT_+ \oplus \curlyT_-$;
moreover we have the
identification $V(\la)_\pm = M(\la)_\pm / \curlyT_\pm$.

We denote the image of $1$ in $V(\lambda)$ by $v_\lambda^+$
when convenient. This module has an induced parity grading from the
associated Verma module by setting $p(v_\lambda^+)=0$.
When considering the image of $1$ in the module $^\omega
V(\lambda)$, we will denote this vector by $v_{\lambda}^-$.

\begin{prop}
Let $M$ be an object of $\curlyC_{\mathrm{int}}$ and let $m\in
M^\lambda$ be a non-zero vector such that $E_im$=0 for all $i$. Then
$\lambda\in X^+$ and there is a unique morphism (in
$\curlyC_{\mathrm{int}}$) $t':V(\lambda)\rightarrow M$ sending
$v_\lambda^+$ to $m$.
\end{prop}
The proof is as in the non-super case \cite[Proposition~3.5.8]{Lu}.

%






\section{The quasi-$\mathcal R$-matrix and the quantum Casimir}
 \label{sec:quasiR}

In this section, we introduce the quasi-$\mathcal R$-matrix as well
as the quantum Casimir for $\UU$ and establish their basic
properties. Using the Casimir element, we show that the category
$\catO_{\mathrm{int}}$ is semisimple and classify its simple object
by dominant integral weights.

\subsection{The quasi-$\mathcal R$-matrix $\Theta$}

Consider the vector spaces
\[
\mathcal{H}_N= \Up\Uz\Big(\sum_{\height\nu\geq N}\Um_\nu\Big)
\otimes \UU+\UU\otimes\Um\Uz\Big(\sum_{\height\nu\geq N}\Up_\nu\Big)
\]
for $N\in \Z_{>0}$. Note that $\mathcal{H}_N$ is a left ideal in
$\UU\otimes \UU$; moreover, for any $u\in \UU\otimes\UU$, we can
find an $r\geq 0$ such that
$\mathcal{H}_{N+r}u\subset\mathcal{H}_{N}$.

Let $(\UU\otimes\UU)^\wedge$ be the inverse limit of the vector
spaces $(\UU\otimes\UU)/\mathcal{H}_n$. Then the $\Qq^\pi$-algebra
structure extends by continuity to a $\Qq^\pi$-algebra structure on
$(\UU\otimes\UU)^\wedge$, and we have the obvious algebra embedding
$\UU\otimes \UU\rightarrow (\UU\otimes\UU)^\wedge$.

Let $\ov{\phantom{x}}:\UU\otimes \UU\rightarrow \UU\otimes \UU$ be
the $\Q$-algebra homomorphism given by
$\ov{\phantom{x}}\otimes\ov{\phantom{x}}$. This extends to a
$\Q$-algebra homomorphism on the completion. Let
$\ov{\Delta}:\UU\rightarrow \UU\otimes\UU$ be the  $\Qq^\pi$-algebra
homomorphism given by $\ov{\Delta}(x)=\ov{\Delta(\ov{x})}$.

\begin{thm}\label{thm:rmatrix}
\begin{enumerate}
\item[(a)]
There is a unique family of elements $\Theta_\nu\in
\Um_\nu\otimes\Up_\nu$ (with $\nu\in\N[I]$) such that
$\Theta_0=1\otimes 1$ and $\Theta=\sum_\nu \Theta_\nu\in
(\UU\otimes\UU)^\wedge$ satisfies
$\Delta(u)\Theta=\Theta\ov{\Delta}(u)$ for all $u\in \UU$ (where
this identity is in $(\UU\otimes\UU)^\wedge$).

\item[(b)]
Let $B$ be a $\Qq^\pi$-basis of $\ff$ such that $B_\nu=B\cap
\ff_\nu$ is a basis of $\ff_\nu$ for any $\nu$. Let $\set{b^*| b\in
B_\nu}$ be the basis of $\ff_\nu$ dual to $B_\nu$ under $(,)$. We
have
\[
\Theta_\nu=(-1)^{\height \nu} \pi^{e(\nu)} \pi_\nu q_{\nu}\sum_{b\in B_\nu}
b^-\otimes b^{*+}\in \Um_\nu\otimes\Up_\nu,
\]
where $e(\nu)$ is defined as in \S \ref{subsec:antipode}.
\end{enumerate}
\end{thm}
The element $\Theta$ will be called the {\em quasi-$\mathcal R$-matrix} for $\UU$.

\begin{proof}
Consider an element $\Theta\in (\UU\otimes\UU)^\wedge$ of the
form $\Theta=\sum_\nu \Theta_\nu$ with $\Theta_\nu=\sum_{b,b'\in
B_\nu} c_{b',b} b'^-\otimes b^{*+}$, $c_{b',b}\in\Qq^\pi$. The set
of $u\in\UU$ such that $\Delta(u)\Theta=\Theta\ov{\Delta}(u)$ is
clearly a subalgebra of $\UU$ containing $\UU^0$. Therefore, it is
necessary and sufficient that it contains the $E_i$ and $F_i$. This
amounts to showing that
\begin{align*}
\sum_{b_1,b_2\in B_\nu}  c_{b_1,b_2}E_ib_1^-\otimes b_2^{*+}
 &+\sum_{b_3,b_4\in B_{\nu-i}} \pi_i^{p(b_3)} c_{b_3,b_4}
  \wtd{J}_i\wtd{K}_{i}b_3^-\otimes E_ib_4^{*+}\\
=\sum_{b_1,b_2\in B_\nu} \pi_i^{p(b_2)} c_{b_1,b_2} b_1^-E_i\otimes b_2^{*+}
&+\sum_{b_3,b_4\in B_{\nu-i}} c_{b_3,b_4} b_3^-\wtd{K}_{-i}\otimes b_4^{*+}E_i, \\
\intertext{and}
\sum_{b_1,b_2\in B_\nu} \pi_i^{p(b_1)} c_{b_1,b_2} b_1^-\otimes F_ib_2^{*+}
 &+\sum_{b_3,b_4\in B_{\nu-i}} c_{b_3,b_4} F_ib_3^-\otimes \wtd{K}_{-i}b_4^{*+}\\
=\sum_{b_1,b_2\in B_\nu} c_{b_1,b_2} b_1^-\otimes b_2^{*+}F_i
 &+\sum_{b_3,b_4\in B_{\nu-i}} \pi_i^{p(b_4)}c_{b_3,b_4}
  b_3^-F_i\otimes b_4^{*+}\wtd{J}_{i}\wtd{K}_{i}.
\end{align*}
Let $z\in \ff$. Then since the inner product is nondegenerate, this
equality is equivalent to the equality
\begin{align*}
\sum_{b_1,b_2\in B_\nu} c_{b_1,b_2} (b_2^*,z) &(E_ib_1^{-}-\pi_i^{p(b_2^*)}b_1^-E_i)\\
+\sum_{b_3,b_4\in B_{\nu-i}} c_{b_3,b_4}
&\parens{\pi_i^{p(b_3)}(\theta_ib_4^*,z)\wtd{J}_i\wtd{K}_{i}b_3^{-}
 - (b_4^*\theta_i,z)b_3^{-}\wtd{K}_{-i}}=0, \\
\intertext{and}
\sum_{b_1,b_2\in B_\nu} c_{b_1,b_2} (b_1,z) &(\pi_i^{p(b_1)} F_ib_2^{*+}-b_2^{*+}F_i)\\
+\sum_{b_3,b_4\in B_{\nu-i}} c_{b_3,b_4}
&\parens{(\theta_ib_3,z)\wtd{K}_{-i}b_4^{*+}-\pi_i^{p(b_4)}
(b_3\theta_i,z)b_4^{*+}\wtd{J}_i\wtd{K}_{i}}=0.
\end{align*}
Note that $p(b_1)=p(b_2)=p(b_3)+p(i)=p(b_4)+p(i)$.
Using Proposition \ref{prop:EFcommutation} and the derivations, we have
\begin{align*}
&\sum_{b_1,b_2\in B_\nu} (\pi_iq_i-q_i^{-1})^{-1}c_{b_1,b_2} (b_2^*,z)
 (\wtd{J}_i\wtd{K}_i\,\ir(b_1)^- -\pi_i^{p(b_1)-p(i)}\ri(b_1)^-\wtd{K}_{-i})\\
+&\sum_{b_3,b_4\in B_{\nu-i}} c_{b_3,b_4}(\theta_i,\theta_i)
 \parens{\pi_i^{p(b_3)}(b_4^*,\ir(z))\wtd{J}_i
 \wtd{K}_{i}b_3^{-}- (b_4^*,\ri(z))b_3^{-}\wtd{K}_{-i}}=0, \\
\intertext{and}
&\sum_{b_1,b_2\in B_\nu} -(\pi_iq_i-q_i^{-1})^{-1}c_{b_1,b_2}
 (b_1,z) (\ri(b_2)^+\wtd{J}_i\wtd{K}_i-\pi_i^{p(b_2)-p(i)}\wtd{K}_{-i}\,\ir(b_2)^+)\\
+&\sum_{b_3,b_4\in B_{\nu-i}} c_{b_3,b_4}(\theta_i,\theta_i)
\parens{(b_3,\ir(z))\wtd{K}_{-i}b_4^{*+}-\pi_i^{p(b_4)}
 (b_3,\ri(z))b_4^{*+}\wtd{J}_i\wtd{K}_{i}}=0.
\end{align*}
Using the triangular decomposition, this is equivalent to the equalities
\begin{align}
\label{eq:qrmc}
\sum_{b_1,b_2}c_{b_1,b_2} (b_2^*,z)\ir(b_1)
 +\sum_{b_3,b_4}\pi_iq_i \pi_i^{p(b_4)} c_{b_3,b_4}(b_4^*,\ir(z))b_3&=0, \\
\label{eq:qrmd}
\sum_{b_1,b_2}c_{b_1,b_2} \pi_i^{p(b_1)-p(i)}(b_2^*,z)\ri(b_1)
 +\sum_{b_3,b_4}\pi_iq_i c_{b_3,b_4}(b_4^*,\ri(z))b_3&=0, \\
\label{eq:qrme}
\sum_{b_1,b_2}c_{b_1,b_2} (b_1,z)\ri(b_2)
 +\sum_{b_3,b_4}\pi_iq_i\pi_i^{p(b_4)}c_{b_3,b_4}(b_3,\ri(z))b_4^{*}&=0, \\
\label{eq:qrmf}
\sum_{b_1,b_2}\pi_i^{p(b_2)-p(i)}c_{b_1,b_2} (b_1,z)\ir(b_2)
 +\sum_{b_3,b_4}\pi_iq_ic_{b_3,b_4}(b_3,\ir(z))b_4^{*}&=0.
\end{align}

Now when $c_{b,b'}=(-1)^{\height(\nu)}\pi^{e(\nu)} \pi_\nu q_\nu\delta_{b,b'}$, we have
\begin{align*}
\sum_{b}\pi^{e(\nu)} q_\nu (b^*,z)\ir(b) -\sum_{b'} \pi^{e(\nu)} \pi_\nu q_{\nu}(b'^*,\ir(z))b'&=0, \\
\sum_{b}\pi^{e(\nu-i)}\pi_\nu q_\nu (b^*,z)\ri(b)
 -\sum_{b'}\pi^{e(\nu-i)} \pi_\nu q_{\nu} (b'^*,\ri(z))b'&=0, \\
\sum_{b}\pi^{e(\nu)} \pi_\nu q_\nu (b,z)\ri(b) -\sum_{b'}\pi^{e(\nu)} \pi_\nu q_{\nu}(b',\ri(z))b'^{*}&=0, \\
\sum_{b}\pi^{e(\nu-i)} \pi_\nu q_\nu (b,z)\ir(b)
 -\sum_{b'}\pi^{e(\nu-i)} \pi_\nu q_{\nu}(b',\ir(z))b'^{*}&=0.
\end{align*}

These equalities are easily verified by checking when $z$ is a basis or dual basis element.

Thus the existence of such a $\Theta$ is verified. Suppose
$\Theta'_\nu$ and $\Theta'$ also satisfy the conditions in (a). Then
$\Theta-\Theta'=\sum c_{b,b'} b^-\otimes b'^+$ must satisfy 
\eqref{eq:qrmc}-\eqref{eq:qrmf}
and has $c_{b,b}=0$ for $b\in B_0$. Suppose $c_{b,b'}=0$ for
$b,b'\in B_\nu'$ for $\height(\nu')  < n$ and assume
$\height(\nu)=n$. Then the second sum in \eqref{eq:qrmc} is zero, so
$\ir(\sum_{b_1,b_2}c_{b_1,b_2} (b_2^*,z)b_1)=0$. But then
$\sum_{b_1,b_2}c_{b_1,b_2} (b_2^*,z)b_1=0$, whence
$(\sum_{b_2}c_{b_1,b_2} b_2^*,z)=0$ for all $z\in \ff$. Therefore
$c_{b_1,b_2}=0$ for all $b_1,b_2\in B_\nu$. By induction
$\Theta-\Theta'=0$ proving uniqueness.
\qed\end{proof}

\begin{example}\label{example:rankonetheta}
Let $I=I_\one={i}$ as in Example \ref{example:rankoneqg},
and let us determine $\Theta$ in this case using Theorem 
\ref{thm:rmatrix}(b).
The obvious basis to choose is $B=\set{\theta^{(n)}:n\in \N}$,
and then we see from Lemma \ref{lem:divpowcoprod} that 
$\Theta=\sum_n a_n F^{(n)}\otimes E^{(n)}$,
where $a_n=(-1)^n (\pi q)^{-\binom{n+1}{2}}[n]^!(\pi q-q^{-1})^n$
(compare with \cite[\S 5]{CW}).
\end{example}

Recall that the bar involution on $\UU$ makes sense under the
assumption that the super Cartan datum is consistent.

\begin{cor}\label{C:thetathetainverse}
Assume that the super Cartan datum is consistent. We have
$\Theta\ov\Theta=\ov\Theta\Theta=1\otimes 1$ with equality in the
completion.
\end{cor}

\begin{proof}
First note that by construction $\Theta$ is invertible. We have
$\Delta(u)\Theta=\Theta\ov\Delta(u)$, so
$\Theta\Delta(\ov{u})=\Theta\ov\Delta(\ov u)=\Theta \ov{\Delta(u)}$.
Now applying the bar involution to both sides and rearranging, we
get\[\ov\Theta^{\,-1} \ov\Delta(u)=  \Delta(u)\ov{\Theta}^{\,-1}.\]
By uniqueness, $\ov{\Theta}^{\,-1}=\Theta$.
\qed\end{proof}

We can specialize the identity
$\Delta(u)\Theta=\Theta\ov{\Delta}(u)$ to deduce
\begin{align*}
(E_i\otimes 1)\Theta_\nu+(\wtd{J}_i\tK_i\otimes E_i)\Theta_{\nu-i}
 &=\Theta_\nu(E_i\otimes 1)+\Theta_{\nu-i}(\tK_{-i}\otimes E_i),
  \\
(1\otimes F_i)\Theta_\nu+(F_i\otimes \tK_{-i})\Theta_{\nu-i}
 &=\Theta_\nu(1\otimes F_i)+\Theta_{\nu-i}(F_i\otimes \wtd{J}_i\tK_{i}).
 \end{align*}
Setting $\Theta_{\leq p}=\sum_{\height \nu\leq p}\Theta_\nu$, we obtain that
\begin{align}
(E_i\otimes 1  +\wtd{J}_i\tK_i\otimes E_i) & \Theta_{\leq p}
-\Theta_{\leq p}(E_i\otimes 1+\tK_{-i}\otimes E_i)\notag
  \\
=& \sum_{\height \nu= p}(\wtd{J}_i\tK_i\otimes
E_i)\Theta_\nu-\sum_{\height \nu= p}\Theta_\nu(\tK_{-i}\otimes E_i),\label{eq:thetacomE}
  \\
(F_i\otimes \tK_{-i}  +1\otimes F_i) & \Theta_{\leq p}-\Theta_{\leq
p}(F_i\otimes \tK_{i}+1\otimes F_i)\notag
  \\
=&\sum_{\height \nu= p}(F_i\otimes \tK_{-i})\Theta_\nu-\sum_{\height
\nu= p}\Theta_\nu(F_i\otimes \wtd{J}_i\tK_{i}).\label{eq:thetacomF}
 \end{align}


%
%
\subsection{The quantum Casimir}

Let $B, B_\nu$ be as in Theorem \ref{thm:rmatrix}. Let $S$ be the
antipode and $\mathbf m:\UU\otimes \UU\rightarrow \UU$ 
be the multiplication map $u\otimes u'\mapsto uu'$. 
Applying $\mathbf m(S\otimes 1)$ to the identities
\eqref{eq:thetacomE} and \eqref{eq:thetacomF}, 
we obtain that, for any $p\geq 0$,
\begin{align*}
&\sum_{\height \nu\leq p}\sum_{b\in B_\nu} (-1)^{\height \nu}\pi_\nu q_\nu
\left( S(E_i b^-)b^{*+}+\pi_i^{p(\nu)}S(\wtd{J}_i\tK_i b^-)E_ib^{*+} \right.
  \\
&\qquad\qquad\qquad\qquad \qquad \left.
-\pi_i^{p(\nu)}S(b^-E_i)b^{*+} -S(b^-\tK_{-i})b^{*+}E_i \right)
 \\
&=\sum_{\height \nu= p}\sum_{b\in B_\nu}(-1)^{p}\pi_\nu q_\nu
\parens{\pi_i^{p(\nu)}S(\wtd{J}_i\tK_i
b^-)E_ib^{*+}-S(b^-\tK_{-i})b^{*+}E_i},
\end{align*}
and
\begin{align*}
&\sum_{\height \nu\leq p}\sum_{b\in B_\nu} (-1)^{\height \nu}\pi_\nu q_\nu
\left( \pi_i^{p(\nu)}S(b^-)F_ib^{*+}+S(F_i b^-)\tK_{-i}b^{*+} \right.
  \\
&\qquad\qquad\qquad\qquad \qquad \left.
-S(b^-)b^{*+}F_i-\pi_i^{p(\nu)}S(b^-F_i )b^{*+}\wtd{J}_i\tK_{i}
\right)
 \\
&=\sum_{\height \nu= p}\sum_{b\in B_\nu}(-1)^{p}\pi_\nu q_\nu
\parens{S(F_i b^-)\tK_{-i}b^{*+}-\pi_i^{p(\nu)}S(b^-F_i
)b^{*+}\wtd{J}_i\tK_{i}}.
\end{align*}

Now set $\Omega_{\leq p}=\sum_{\height \nu\leq p}\sum_{b\in
B_\nu}(-1)^{\height \nu} \pi_\nu q_\nu S(b^-)b^{*+}$. Then observing
that
\begin{align*}
S(E_i b^-)b^{*+} &+\pi_i^{p(\nu)}S(\wtd{J}_i\tK_i b^-)E_ib^{*+}
 \\
 &=\pi_i^{p(\nu)}S(b^-)(-\wtd{J}_{-i}\tK_{-i}E_i) b^{*+}
+\pi_i^{p(\nu)}S(b^-)\wtd{J}_{-i}\tK_{-i}E_ib^{*+}=0,
\end{align*}
we have
\begin{align*}
\wtd{J}_{-i} & \tK_{-i}E_i\Omega_{\leq p}-\tK_{i}\Omega_{\leq p}E_i
 \\
&=\sum_{\height \nu= p}\sum_{b\in B_\nu}(-1)^{p}\pi_\nu q_\nu
\parens{\pi_i^{p(\nu)}S(\tK_i
b^-)E_ib^{*+}-S(b^-\tK_{-i})b^{*+}E_i},
  \\
\Omega_{\leq p} & F_i-F_i\tK_{i}\Omega_{\leq p}\wtd{J}_{i}\tK_{i}
 \\ &
=-\sum_{\height \nu= p}\sum_{b\in B_\nu}(-1)^{p}\pi_\nu q_\nu
\parens{S(F_i b^-)\tK_{-i}b^{*+}-\pi_i^{p(\nu)}S(b^-F_i
)b^{*+}\wtd{J}_i\tK_{i}}.
\end{align*}

\begin{example}
Let $I=I_1={i}$ as in Examples \ref{example:rankoneqg} 
and let $\Theta$ be as defined in Example~ \ref{example:rankonetheta}.
Then using \S \ref{subsec:antipode},
\[\Omega_{\leq p}=
\sum_{1\leq n\leq p} a_n (\pi q^2)^{-\binom{n}{2}}F^{(n)}K^n E^{(n)}.\]
We note that though this is a rather different construction than the 
Casimir-type element in \cite{CW}, it will nevertheless be used
toward a similar purpose.
\end{example}

Let $M\in \catO$. Then for any $m\in M$ we have that
$\Omega(m)=\Omega_{\leq p} m$ is independent of $p$ when $p$ is
large enough. We can write 
\[\Omega(m)=\sum_{b} (-1)^{\height
|b|}\pi_{|b|}q_{|b|}S(b^-)b^{*+} m\] 
Then we have
\[\label{eq:omegacomm} \wtd{J}_{-i}\tK_{-i}E_i\Omega=\tK_i\Omega E_i,
 \qquad \Omega F_i=F_i\tK_i\Omega \wtd{J}_i\tK_i,
  \qquad \Omega K_\mu = K_\mu
\Omega,
\]
as operators on $M$. Therefore for $m\in M^\lambda$, we have
$$
\Omega E_im=(\pi_iq_i^{2})^{-\ang{i, \lambda+i'}}E_i\Omega m, \qquad
\Omega F_i=(\pi_iq_i^2)^{\ang{i, \lambda+i'}}F_i\Omega m.
$$

This can be rephrased in terms of the antipode.
Define the $\Qq^\pi$-linear map $\ov{S}:\UU\rightarrow
\UU$ by $\ov{S}(u)=\ov{S(\ov{u})}$. Then $\Omega
\ov{S}(u)=S(u)\Omega:M\rightarrow M$ for $u\in \UU$.

Let $C$ be a fixed coset of $X$ with respect to $\Z[I]\leq X$. Let
$G:C\rightarrow \Z$ be a function such that
\begin{equation}
\label{eq:Gdef}
G(\lambda)-G(\lambda-i')=\frac{i\cdot i}2 \ang{i,\lambda}
\quad \text{ for all } \lambda\in C,i\in I.
\end{equation}
Clearly such a function exists
and is unique up to addition of a constant function.

\begin{lem}\label{lem:Gcomparison}
Let $\lambda, \lambda'\in C\cap X^+$. If $\lambda\geq \lambda'$ and
$G(\lambda)=G(\lambda')$, then $\lambda=\lambda'$.
\end{lem}

Let $M\in \curlyC$. For each $\Z[I]$-coset $C$ in $X$, define
$M_C=\bigoplus_{\lambda\in C} M^\lambda$. It is clear that
\begin{equation}\label{eq:cosetdecomp}
M=\bigoplus_{C\in X/\Z[I]} M_C.
\end{equation} 

\begin{prop} \label{prop:casimirprops}
Let $M\in\catO$, and let $\Omega:M\rightarrow M$ be as above.
\begin{enumerate}
\item[(a)] Assume there exists $C$ as above such that $M=M_C$.
Let $G:C\rightarrow \Z$ be a function satisfying
\eqref{eq:Gdef}.
We define a linear map $\Xi:M\rightarrow M$ by $\Xi(m)=(\pi q^2)^{G(\lambda)} m$
for all $\lambda\in C$ and $m\in M^\lambda$.
Then $\Omega\Xi$ is a locally finite $U$-module homomorphism.

\item[(b)] Assume that $M$ is a quotient of $M(\lambda')$.
Then $\Omega\Xi$ acts as $(\pi q^2)^{G(\lambda')}$ on $M$.

\item[(c)]
For $M$ as in (a), the eigenvalues of $\Omega\Xi$ are of the form
$(\pi q^2)^c$ for $c\in\Z$.
\end{enumerate}
\end{prop}

The operator $\Omega\Xi$ is called the {\em Casimir element} of $\UU$
(though note that the Casimir element formally lives in a completion of $\UU$).

\begin{proof}
We compute that for $m\in M^\lambda$,
\begin{align*}
\Omega\Xi E_im&=\Omega (\pi q^2)^{G(\lambda+i')} E_i m
=
(\pi q^2)^{G(\lambda+i')-G(\lambda)-s_i\ang{i,\lambda+i'}} E_i\Omega\Xi m=E_i\Omega\Xi m.
\end{align*}

A similar argument applies to the $F_i$, and clearly $\Omega\Xi$ commutes with
$K_\mu$, $J_\mu$ proving the first assertion of (a).
The local finiteness claim is a standard category $\mathcal O$ type argument.
Parts (b) and (c) follow now easily.
\qed\end{proof}

\subsection{The complete reducibility in $\catO_{\mathrm{int}}$}

Recall the categories  $\catO$ and $\curlyC_{\mathrm{int}}$ from
\S \ref{subsec:catCO}-\ref{subset:intO}. Form another category $\catO_{\mathrm{int}}
:=\catO\cap \curlyC_{\mathrm{int}}$.

\begin{lem}  \label{lem:simplezero}
Let $M\in\curlyC$. Assume that $M$ is a nonzero quotient of the
Verma module $M(\lambda)$ and that $M$ is integrable. Then
\begin{enumerate}
\item[(a)]
 $\lambda\in X^+$;
 \item[(b)]
 $M_+$ and $M_-$ are either simple or zero.
 \end{enumerate}
\end{lem}

\begin{proof}
It is clear that (a) holds by some rank one consideration.
An argument similar to that for \cite[Lemma~6.2.1]{Lu} shows that
if $\dim_\Qq M^\la=1$ then $M$ is simple; in this case, $M$ must be
equal to either $M_+$ or $M_-$. Otherwise, $\dim_\Qq M^\la=2$,
then $\dim_\Qq M_+^\la =\dim_\Qq M_1^\la =1$, and we repeat
the argument above for the integrable $\UU$-module $M_\pm$.
\qed\end{proof}

\begin{thm}
Let $M$ be a $\UU$-module in $\catO_{\mathrm{int}}$. Then $M$ is a
sum of simple $\UU$-submodules.
\end{thm}

\begin{proof}

Note that as discussed in \S \ref{subsec:catCO} we may assume that $M=M_+$ or $M=M_-$.
Since the case for $M_+$ follows from \cite[Theorem~6.2.2]{Lu}, it is
enought to prove the theorem for $M=M_-$. Virtually the same argument
as in {\em loc. cit.} holds, which we will now sketch.

Using \eqref{eq:cosetdecomp},
we may further assume there is a coset $C$ of $\Z[I]$ in $X$
such that $M=M_C$. Then we may pick a 
function $G$ satisfying \eqref{eq:Gdef}
and avail ourselves of Proposition \ref{prop:casimirprops}.
Since the Casimir element commutes with the $\UU$-action, we may further
assume that $M$ lies in a generalized eigenspace of the Casimir element.

Consider the set of singular vectors of $M$(that is, the set
of vectors $m\in M$ for which $E_im=0$ for all $i\in I$)
and let $M'$ be the submodule they generate.
Then each homogeneous singular vector generates a simple
submodule by virtue of Lemma \ref{lem:simplezero},
so $M'$ is a sum of simple modules. 

It remains to show that
$M=M'$, so take $M''=M/M'$ and suppose $M''\neq 0$.
Then there is a maximal weight $\lambda\in C$ such that
$M''^\lambda\neq 0$. Then the Casimir element acts on 
the submodule generated by a nonzero 
$m_1\in M''^\lambda$ by $(-q^2)^{G(\lambda)}$ 
by Proposition \ref{prop:casimirprops},
and so in particular $M$ must lie in the generalized
$(-q^2)^{G(\lambda)}$-eigenspace of the Casimir element.

On the other hand, $m$ is the image of a vector $\tilde{m}\in M\setminus M'$.
The $\Up$-module generated by $\tilde m$ contains a singular vector $m_2$ 
of weight $\eta\geq \lambda$, and the Casimir element acts on the module
generated by $m_2$ as $(-q^2)^{G(\eta)}$. Then $G(\eta)=G(\lambda)$
and $\eta\geq \lambda$, so by Lemma \ref{lem:Gcomparison}
$\eta=\lambda$. But the $\tilde{m}$ is a singular vector, contradicting
that our choice of $m_1$ was nonzero.
\qed\end{proof}

\begin{cor}
\begin{enumerate}
\item[(a)]
For $\lambda\in X^+$, the $\UU$-modules $V(\lambda)_+$ and $V(\la)_-$ are simple
objects of $\catO_{\mathrm{int}}$.

\item[(b)]
For $\lambda,\lambda'\in X^+$, the $\UU$-modules $V(\lambda)_+$ and
$V(\lambda')_+$, and respectively $V(\lambda)_-$ and
$V(\lambda')_-$, are isomorphic if and only if $\lambda=\lambda'$.
(Clearly, $V(\la)_+$ and $V(\la')_-$ are non-isomorphic.)

\item[(c)]
Any integrable module in $\catO$ is a direct sum of simple modules
of the form $V(\lambda)_\pm$ for various $\lambda\in X^+$.
\end{enumerate}
\end{cor}

\begin{proof}
The argument in \cite[Corollary~6.2.3]{Lu} holds using our Lemma~\ref{lem:simplezero} above.
\qed\end{proof}

\subsection{Character formula}
\label{subsec:char}

Denote by $\rho \in X$ such that $\langle i,\rho \rangle =1$ for all
$i\in I$. We claim the following character formula of $V(\la)$
for {\em every} $\la\in X^+$:
\begin{align}
\label{eq:chform}
\text{ch}\, V(\la)_\pm = \frac{\sum_{w\in W} (-1)^{\ell(w)}
e^{w(\la+\rho)-\rho}}{\sum_{w\in W} (-1)^{\ell(w)} e^{w(\rho)-\rho}
}.
\end{align}
This is equivalent to claiming $V(\la)$ is always a $\Qq^\pi$-free module for each $\la \in X^+$.
This character formula holds for $V(\la)_+$ with $\la\in X^+$ by a theorem of Lusztig \cite{Lu1}. A proof of this formula for $V(\la)_-$ is possible, but requires techniques outside the scope of this paper.

Assume now that $\la \in X^+$ satisfies an {\em evenness condition}
\begin{align}
\label{eq:evencond}
\langle i, \la \rangle \in 2\Z_+, \quad \forall i\in I_\one.
\end{align}
Then the action of $\UU$ on $V(\la)$ factors through an action of
the algebra  $\UU/\mathcal J$ 
(see \S\ref{subsec:specializations}), and \eqref{eq:chform}
holds by \cite[Theorem~4.9]{BKM}  on the characters of integrable modules of the usual
quantum groups.
The irreducible integrable modules of the corresponding Kac-Moody
superalgebras were known \cite{Kac} to be parametrized by highest
weights $\la \in X^+$ satisfying \eqref{eq:evencond}. Hence, for $\la \in
X^+$ which does not satisfy \eqref{eq:evencond}, 
the usual $q$-deformation argument
cannot be applied directly to $V(\la)_-$.


Note there are always weights $\la$ satisfying \eqref{eq:evencond} 
which are large
enough relative to every $i\in I$. Therefore, the same type of
arguments as in \cite[Chapter 33]{Lu} show that the algebra $\bf f$
and hence $\UU$ admit the following equivalent formulations.

\begin{prop}
 \label{prop:f=Serre}
The algebra $\bf f$ is isomorphic to the algebra generated by
$\theta_i, i\in I$, subject to the quantum Serre relation as in
Proposition~\ref{prop:quantserre}.
\end{prop}

\begin{prop}
The algebra $\UU$ is isomorphic to the algebra generated by $ E_i,
F_i$ $(i\in I)$ and $J_{\mu}, K_\mu$ $(\mu\in Y),$ subject to the
relations~\ref{sec:Udef}(a)-(f) and the quantum Serre relations for
$E_i$'s as well as for $F_i$'s (in place of $\theta_i$'s in
Proposition~\ref{prop:quantserre}).
\end{prop}

As a consequence of \eqref{eq:chform} and
Proposition~\ref{prop:f=Serre}, the character of $\UU^-$ is given by
\begin{align}
\text{ch}\, \UU^- = \frac{1}{\sum_{w\in W} (-1)^{\ell(w)}
e^{w(\rho)-\rho} } =\prod_{\alpha >0}
(1-e^{-\alpha})^{(-1)^{1+p(\alpha)} \dim \mathfrak g_{\alpha}},
\end{align}
where $\mathfrak g$ denotes the Kac-Moody superalgebra of type
$(I,\cdot)$ (cf. \cite{Kac}), ``$\alpha>0$" denotes positive roots of
$\mathfrak g$, $p(\cdot)$ denotes the parity function, and
$\mathfrak g_\alpha$ denotes the $\alpha$-root space.

%
%

\section{Higher Serre relations}
\label{sec:higherSerre}

In this section we formulate and establish the higher Serre
relations, which will be instrumental in determining the action of a
braid group on a quantum covering group and integrable modules in a future work.

\subsection{Higher Serre elements}
 \label{subsec:higher-elt}

For $i,j\in I$, and $n,m\geq 0$, set
\begin{align*}
p(n,m;i,j)=mnp(i)p(j)+{m\choose 2}p(i).
\end{align*}

For $i\neq j$, define the elements
\begin{align}
\label{E:eijnm}
 e_{i,j;n,m}
&=\sum_{r+s=m}(-1)^r\pi_i^{p(n,r;i,j)}(\pi_iq_i)^{-r(n\ang{i,j'}+m-1)}E_i^{(r)}E_j^{(n)}E_i^{(s)},
 \\
\label{E:e'ijnm}
 e'_{i,j;n,m}
&=\sum_{r+s=m}(-1)^r\pi_i^{p(n,r;i,j)}(\pi_iq_i)^{-r(n\ang{i,j'}+m-1)}E_i^{(s)}E_j^{(n)}E_i^{(r)},
 \\
\label{E:fijnm}
f_{i,j;n,m}
&=\sum_{r+s=m}(-1)^r\pi_i^{p(n,r;i,j)}q_i^{r(n\ang{i,j'}+m-1)}F_i^{(s)}F_j^{(n)}F_i^{(r)},
 \\
\label{E:f'ijnm}
f'_{i,j;n,m}
&=\sum_{r+s=m}(-1)^r\pi_i^{p(n,r;i,j)}q_i^{r(n\ang{i,j'}+m-1)}F_i^{(r)}F_j^{(n)}F_i^{(s)}.
\end{align}
When there is no confusion by fixing $i$ and $j$, we will abbreviate
$e_{i,j;n,m}=e_{n,m}$, $e'_{i,j;n,m}=e'_{n,m}$,
$f_{i,j;n,m}=f_{n,m}$, $f'_{i,j;n,m}=f'_{n,m}$.  Note that we have
the equalities
\begin{align}\label{eq:higherserreids}
e'_{n,m}=\sigma(e_{n,m}),\;\;\;f'_{n,m}=\sigma\omega^2(f_{n,m}),\;\;\;
e_{n,m}=\om(\overline{f'_{n,m}}),\quad
e'_{n,m}=\om(\overline{f_{n,m}}).
\end{align}

\subsection{Commutations with divided powers}

\begin{lem}\label{L:HigherSerre1} The following hold:
\begin{enumerate}
\item[(a)]
$\displaystyle
-q_i^{-n\ang{i,j'}-2m}\pi_i^{m+np(j)}E_ie_{n,m}+e_{n,m}E_i=[m+1]_ie_{n,m+1}$.
\item[(b)]
$\displaystyle
-F_ie_{n,m}+\pi_i^{m+np(j)}e_{n,m}F_i=[-n\ang{i,j'}-m+1]_i\pi_i^{np(j)+1}\wtd{K}_i^{-1}e_{n,m-1}$.
\end{enumerate}
\end{lem}

\begin{proof}
When $i\in I_\zero$ this is \cite[Lemma 7.1.2]{Lu}. We therefore
assume $i\in I_\one$. Then, $\ang{i,j'}\in 2\Z$ by
\ref{subsec:Cartan}(d). The left hand side of (a) is
\begin{align*}
&\sum_{r+s=m}(-1)^r\pi_i^{p(n,r;i,j)}
(\pi_iq_i)^{-r(n\ang{i,j'}+m-1)}[s+1]_i
 E_i^{(r)}E_j^{(n)}E_i^{(s+1)}\\
&\hspace{.8in}+\sum_{r+s=m}(-1)^{r+1}\pi_i^{p(n,r;i,j)+np(j)+m}(\pi_iq_i)^{-r(n\ang{i,j'}+m-1)
 -n\ang{i,j'}-2m}\\
&\hspace{.8in}\times [r\!+\!1]_i
 E_i^{(r+1)}E_j^{(n)}E_i^{(s)}\\
&=\sum_{r+s=m+1}(-1)^r\pi_i^{p(n,r;i,j)}(\pi_iq_i)^{-r(n\ang{i,j'}+m)}
 \\
 &\hspace{.8in} \times \left((\pi_iq_i)^{r-(m+1)}
 \pi_i^{r-1+m}[r]_i+(\pi_iq_i)^r[s]_i \right)E_i^{(r)}E_j^{(n)}E_i^{(s)},
\end{align*}
where we have used
\begin{align}\label{eq:pequality}
p(n,r-1;i,j)\equiv p(n,r;i,j)+np(i)p(j)+(r-1)p(i)\;\;\; (\text{mod }
2)
\end{align}
in the last line. Part (a) now follows from the computation
$$(\pi_iq_i)^{r-(m+1)}\pi_i^{r-1+m}[r]_i+(\pi_iq_i)^r[s]_i
=q_i^{-s}[r]_i\!+\!(\pi_iq_i)^r[s]_i=[r+s]_i=[m+1]_i.$$

To prove (b), observe that
\begin{align*}
F_iE_i^{(r)}E_j^{(n)}E_i^{(s)}=&\pi_i^{r+np(j)+s}
E_i^{(r)}E_j^{(n)}E_i^{(s)}\!F_i
\!-\!\pi_i^{r+np(j)+1}E_i^{(r)}E_j^{(n)}E_i^{(s-1)}[\wtd{K}_i;s\!-\!1]\\
    &
-\pi_iE_i^{(r-1)}[\wtd{K}_i;r-1]E_j^{(n)}E_i^{(s)}.
\end{align*}
Therefore,
\begin{align*}
-F_i & e_{n,m}+\pi_i^{m+np(j)}e_{n,m}F_i\\
 =&\sum_{r+s=m}(-1)^{r}\pi_i^{p(n,r;i,j)+r+np(j)+1}(\pi_iq_i)^{-r(n\ang{i,j'}+m-1)}
  E_i^{(r)}E_j^{(n)}E_i^{(s-1)}[\wtd{K}_i;s\!-\!1]\\
 &+ \sum_{r+s=m}(-1)^{r}\pi_i^{p(n,r;i,j)+1}(\pi_iq_i)^{-r(n\ang{i,j'}+m-1)}
 E_i^{(r-1)}[\wtd{K}_i;r-1]E_j^{(n)}E_i^{(s)}\\
 =&\sum_{r+s=m-1}(-1)^{r}\pi_i^{p(n,r;i,j)+r+np(j)+1}(\pi_iq_i)^{-r(n\ang{i,j'}+m-1)}
 E_i^{(r)}E_j^{(n)}E_i^{(s)}[\wtd{K}_i;s]\\
 &+ \sum_{r+s=m-1}(-1)^{r-1}\pi_i^{p(n,r+1;i,j)+1}(\pi_iq_i)^{-(r+1)(n\ang{i,j'}+m-1)}
 E_i^{(r)}[\wtd{K}_i;r]E_j^{(n)}E_i^{(s)}\\
 =&\sum_{r+s=m-1}(-1)^{r}\pi_i^{p(n,r;i,j)}(\pi_iq_i)^{-r(n\ang{i,j'}+(m-1)-1)}
  \pi_i^{np(j)+1}\pi_i^r(\bigstar)
  E_i^{(r)}E_j^{(n)}E_i^{(s)}
\end{align*}
where, using \eqref{eq:pequality} we compute
\begin{align*}
(\bigstar)=&(\pi_iq_i)^{-r}[\wtd{K}_i;-s-n\ang{i,j'}-2r]
  -(\pi_iq_i)^{-n\ang{i,j'}-m+1-r}[\wtd{K}_i;-r]\\
    =&\frac{(\pi_iq_i)^{-r}(\pi_iq_i)^{-m+1-n\ang{i,j'}-r}
    -(\pi_iq_i)^{-n\ang{i,j'}-m+1-r}(\pi_iq_i)^{-r}}{\pi_iq_i-q_i^{-1}}\wtd{J}_i\wtd{K}_i\\
     &\hspace{.5in}+\frac{(\pi_iq_i)^{-n\ang{i,j'}-m+1-r}q_i^r
     - (\pi_iq_i)^{-r}q_i^{n\ang{i,j'}+m-1+r}}{\pi_iq_i-q_i^{-1}}\wtd{K}_i^{-1}\\
     =&\pi_i^r\frac{(\pi_iq_i)^{-n\ang{i,j'}-m+1}-q_i^{n\ang{i,j'}
      +m-1}}{\pi_iq_i-q_i^{-1}}\wtd{K}_i^{-1}.
\end{align*}
This proves (b).
\qed\end{proof}

The next result, which is a $\pi$-analogue of
\cite[Lemma~7.1.3]{Lu}, follows by a straightforward induction
argument.

\begin{lem}\label{L:HigherSerre2} The following formulas hold:
\begin{enumerate}
\item[(a)]
\begin{align*}
\displaystyle
&E_i^{(N)}e_{n,m}=\sum_{k=0}^N(-1)^kq_i^{N(n\ang{i,j'}+2m)+(N-1)k}\pi_i^{N(np(j)+m)+{k\choose
2}}\\
&\hspace{1in}\times\left[{m+k\atop k}\right]_ie_{n,m+k}E_i^{(N-k)};
\end{align*}
\item[(b)]
\begin{align*}
&\displaystyle F_i^{(M)}e_{n,m} =
\sum_{h=0}^M(-1)^hq_i^{-(M-1)h}\pi_i^{M(m+np(j))+(M-m)h}
 \\
 &\hspace{1in}\times\left[{-n\ang{i,j'}-m+h\atop
h}\right]_i\wtd{K}_i^{-h}e_{n,m-h}F_i^{(M-h)}.
\end{align*}
\end{enumerate}
\end{lem}

\begin{lem}\label{L:HigherSerre3} Let $m=1-n\ang{i,j'}$. Then
$$
F_je_{n,m}-\pi_j^{mp(i)+n}e_{n,m}F_j\!=\!
\pi_j^n\!\left(\!\wtd{K}_j^{-1}\frac{q_j^{n-1}}{\pi_jq_j-q_j^{-1}}{e_{n-1,m}}-\wtd{J}_j\wtd{K}_j
\frac{q_j^{1-n}}{\pi_jq_j-q_j^{-1}}\overline{e_{n-1,m}}\!\right).
$$
\end{lem}

\begin{proof}
To begin, if $r+s=m$, then
$$
F_jE_i^{(r)}E_j^{(n)}E_i^{(s)} =\pi_j^{mp(i)+n}E_i^{(r)}E_j^{(n)}
E_i^{(s)}F_j-\pi_j^{rp(i)+1}E_i^{(r)}E_j^{(n-1)}[\wtd{K}_j,n-1]E_i^{(s)}.
$$
Since $m=1-\ang{i,j'}$, the exponent of $\pi_iq_i$ in $e_{n,m}$ is
0; see \eqref{E:eijnm}.  Therefore,
\begin{align*}
&F_je_{n,m}-\pi_j^{mp(i)+n}e_{n,m}F_j\\
 &=-\pi_j\sum_{r+s=m}(-1)^r\pi_i^{p(n-1,r;i,j)}
 [\wtd{K}_j;1-n-ra_{ji}]E_i^{(r)}E_j^{(n-1)}E_i^{(s)}\\
 &=-\pi_j\sum_{r+s=m}(-1)^r\pi_i^{p(n-1,r;i,j)}
 q_i^{-r\ang{i,j'}}\frac{(\pi_jq_j)^{1-n}}{\pi_jq_j-q_j^{-1}}
 \wtd{J}_j\wtd{K}_jE_i^{(r)}E_j^{(n-1)}E_i^{(s)}\\
 &\hspace{.25in}+\pi_j\sum_{r+s=m}(-1)^r
 \pi_i^{p(n-1,r;i,j)}(\pi_iq_i)^{r\ang{i,j'}}
 \frac{(\pi_jq_j)^{n-1}}{\pi_jq_j-q_j^{-1}}\wtd{K}_j^{-1}
 E_i^{(r)}E_j^{(n-1)}E_i^{(s)}.
\end{align*}
We have used $p(n,r;i,j)=p(n-1,r;i,j)+rp(i)p(j)$ to simplify the
second line, and $q_i^{-r\ang{i,j'}}=q_j^{-r\ang{j,i'}}$ and
$\pi_j^{r\ang{j,i'}}=1=\pi_i^{r\ang{i,j'}}$ in the two subsequent
lines. Since $(n-1)\ang{i,j'}+m-1=-\ang{i,j'}$, the result follows.
\qed\end{proof}

As a consequence of the previous lemmas we obtain a generalization
of the quantum Serre relations.

\begin{prop}[Higher Serre Relations] \label{L:HigherSerre}
Let $i,j\in I$ be distinct.
If $m>-n\ang{i,j'}$, then $e_{i,j;n,m}=0$.
\end{prop}

\begin{proof}
As before, fix $i$ and $j$ and write $e_{n,m}=e_{i,j;n,m}$. Note
that $e'_{1,1-\ang{i,j'}}=\sigma(e_{1,1-\ang{i,j'}})$ is just the usual quantum Serre relations
(see Proposition~\ref{prop:quantserre}). Using Lemma~
\ref{L:HigherSerre1}(a), it follows by induction on $m$ that
$e_{1,m}=0$ for $m\geq1-\ang{i,j'}$. Now, let $n>1$ and assume that
$e_{n-1,m}=0$ for all $m>(1-n)\ang{i,j'}$. By Lemma~
\ref{L:HigherSerre1}(b), $e_{n,1-n\ang{i,j'}}$ supercommutes with
$F_i$, and by Lemma \ref{L:HigherSerre3} and induction, it
supercommutes with $F_j$ (note that
$m=1-n\ang{i,j'}>(1-n)\ang{i,j'}$). It trivially supercommutes with
$F_k$ for $k\neq i,j$. Therefore, by Proposition~\ref{prop:x=0} we
deduce that $e_{n,1-n\ang{i,j'}}=0$. Again, using
Lemma \ref{L:HigherSerre1}(a) and induction $m$ the $e_{n,m}=0$ for $m\geq
1-n\ang{i,j'}$.
\qed\end{proof}

\end{document}